\documentclass[a4paper]{article}
\usepackage[english]{babel}
\usepackage[utf8]{inputenc}
\usepackage[left=2.3cm, right=2.3cm, top=1.9cm, bottom=2.0cm,headheight=64pt]{geometry} 

\usepackage{import} 

\usepackage{graphicx} 

\usepackage{amsmath, amsthm, tikz, tikz-cd, amssymb, mathtools }
\usepackage{dsfont, stmaryrd, bbold, yfonts, wasysym} 

\usepackage{hyperref}
\hypersetup{
    colorlinks=true,
    linkcolor=cyan,
    filecolor=magenta,      
    urlcolor=cyan,
    citecolor=teal,
}
\usepackage[capitalise]{cleveref}

\usepackage{authblk, titlesec}
\usepackage{csquotes} 

\usepackage{dirtytalk} 

\usepackage{epigraph}

\usepackage[titletoc]{appendix} 

\usepackage{cite, url}

\usepackage{tikz, tikz-cd}
\usepackage{dsfont}
\usetikzlibrary{matrix}
\usetikzlibrary{arrows}
\usepackage{capt-of}

\theoremstyle{definition}
\crefname{subsection}{subsection}{subsections}
\crefname{equation}{}{}
\crefname{section}{section}{sections}

\newtheorem{theorem}{Theorem}[section]
\newtheorem{example}[theorem]{Example}

\newtheorem{remark}[theorem]{Remark}
\newtheorem{definition}[theorem]{Definition}
\newtheorem{lemma}[theorem]{Lemma}
\newtheorem{corollary}[theorem]{Corollary}
\newtheorem{proposition}[theorem]{Proposition}

\newcommand{\R}{\mathbb{R}}
\newcommand{\C}{\mathbb{C}}
\newcommand{\Z}{\mathbb{Z}}
\newcommand{\N}{\mathbb{N}}
\newcommand{\Q}{\mathbb{Q}}

\newcommand{\A}{\mathds{A}}
\newcommand{\Sp}{\mathds{S}}

\DeclareMathOperator{\MSL}{MSL}
\DeclareMathOperator{\MGL}{MGL}
\DeclareMathOperator{\MU}{MU}
\DeclareMathOperator{\MSU}{MSU}
\DeclareMathOperator{\SU}{SU}
\DeclareMathOperator{\SL}{SL}
\DeclareMathOperator{\SH}{SH}
\DeclareMathOperator{\eff}{eff}
\DeclareMathOperator{\BSL}{BSL}
\DeclareMathOperator{\BGL}{BGL}
\DeclareMathOperator{\HZ}{H\Z}

\DeclareMathOperator{\Preshv}{\mathcal{P}}

\DeclareMathOperator{\Th}{th}
\DeclareMathOperator{\GL}{GL}

\DeclareMathOperator{\Taut}{\mathcal{T}}
\DeclareMathOperator{\Gr}{Gr}
\DeclareMathOperator{\SGr}{SGr}
\DeclareMathOperator{\T}{T}
\DeclareMathOperator{\Hom}{Hom}

\DeclareMathOperator{\SmS}{Sm}

\DeclareMathOperator{\Spc}{\mathcal{S}pc}
\DeclareMathOperator{\thom}{th}
\DeclareMathOperator{\Gm}{\mathds{G}_m}
\makeatletter
\newcommand{\colim@}[2]{%
  \vtop{\m@th\ialign{##\cr
    \hfil$#1\operator@font colim$\hfil\cr
    \noalign{\nointerlineskip\kern1.5\ex@}#2\cr
    \noalign{\nointerlineskip\kern-\ex@}\cr}}%
}
\newcommand{\colim}{%
  \mathop{\mathpalette\colim@{\rightarrowfill@\textstyle}}\nmlimits@
}
\makeatother
\usepackage{yfonts}
\makeatletter
\newcommand{\lim@}[2]{%
  \vtop{\m@th\ialign{##\cr
    \hfil$#1\operator@font lim$\hfil\cr
    \noalign{\nointerlineskip\kern1.5\ex@}#2\cr
    \noalign{\nointerlineskip\kern-\ex@}\cr}}%
}
\newcommand{\dirlim}{%
 \mathop{\mathpalette\lim@{\leftarrowfill@\textstyle}}\nmlimits@
}

\makeatother
\usepackage{yfonts}
\makeatletter
\newcommand{\holim@}[2]{%
  \vtop{\m@th\ialign{##\cr
    \hfil$#1\operator@font holim$\hfil\cr
    \noalign{\nointerlineskip\kern1.5\ex@}#2\cr
    \noalign{\nointerlineskip\kern-\ex@}\cr}}%
}
\newcommand{\holim}{%
  \mathop{\mathpalette\holim@{\leftarrowfill@\textstyle}}\nmlimits@
}

\makeatother
\usepackage{yfonts}
\makeatletter
\newcommand{\hocolim@}[2]{%
  \vtop{\m@th\ialign{##\cr
    \hfil$#1\operator@font hocolim$\hfil\cr
    \noalign{\nointerlineskip\kern1.5\ex@}#2\cr
    \noalign{\nointerlineskip\kern-\ex@}\cr}}%
}
\newcommand{\hocolim}{%
  \mathop{\mathpalette\hocolim@{\rightarrowfill@\textstyle}}\nmlimits@
}
\newcommand{\hocofib}{\operatorname{hocofib}}

\newcommand{\cofib}{\operatorname{cofib}}
\newcommand{\ncolim}{\operatorname{colim}}
\newcommand{\nhocolim}{\operatorname{hocolim}}
\newcommand{\nholim}{\operatorname{holim}}
\newcommand{\nlim}{\operatorname{lim}}

\title{An interpolation between special linear and general algebraic cobordism $\text{MSL}$  and $\text{MGL}$}
\author{Ahina Nandy\thanks{\href{mailto:ahina.nandy@uni-osnabrueck.de}{ahina.nandy@ru.nl}}}
\affil{Radboud Universiteit, Nijmegen}
\date{\vspace{-8ex}}
\begin{document}
\maketitle
\begin{abstract}
 Conner and Floyd determined the torsion in the special unitary bordism $\MSU$ back in the late 1960s. One of the ingredients of their work was an interpolation between $\MSU$ and unitary bordism $\MU$. In this work, we prove an exactly similar relation between the special, and general linear algebraic cobordism $\MSL$, and $\MGL$ in the stable $\A^1$-homotopy category over any Noetherian base scheme of finite Krull dimension. This gives a filtration of $\MGL$ in terms of $\MSL$, and the skeletal filtration of $\mathds{P}^{\infty}$. Using that, we compute the first slice of $\MSL$ over characteristic zero fields. We also compute the first line of the stable homotopy groups of $\MSL$ over fields of characteristic zero.    
\end{abstract}
\tableofcontents
\section{Introduction}
\subsection*{Some stories from topology}
The second half of the twentieth century has seen a spectacular development in the theories of bordism and cobordism. One of the most significant achievements was the identification of the complex (or unitary) bordism ring $\Omega^{U}$ with the graded integral polynomial ring $\Z[a_i | i\geq 1]$, with one generator in every even degree, by Milnor \cite{Milnor}, and Novikov \cite{novikovFirst, novikov}. Later, Quillen's \cite{Quillen} construction of an explicit isomorphism from $\Omega^{U}$ to the Lazard ring, and identification of the $\MU$-formal group law as the universal one, has changed the course of homotopy theory. The notion of complex oriented cohomology theories and the universal complex oriented theory, complex cobordism ($\MU$) has a center stage in classical stable homotopy theory.\\[1.00000ex]
Special unitary ($\SU$)-orientation is a more subtle form of orientation of cohomology theories, and special unitary cobordism $\MSU$ is the universal one among the theories with such orientation. There is a recent resurgence in studying $\SU$-manifolds and $\MSU$ bordism in connection to the study of mirror symmetry \cite{Mirror}. Throughout the 1960s, different approaches to understanding $\MSU$-bordism showed the limitations of the available computational tools and techniques. The coefficient ring $\Omega^{SU}$ is not a polynomial ring. Unlike $\Omega^{U}$, $\MSU_*$ has torsion. Novikov \cite[Appendix 1]{novikov} showed that after inverting $2$, it becomes a polynomial ring. We have 
\[\Omega^{SU} \otimes \Z[1/2] \cong \Z[1/2][x_i | i \geq 2]\]
Here each $x_i$ has degree $2i$. Conner-Floyd determined the torsion in $SU$-bordism \cite{SU_Bordism}. Wall \cite{Wall_1966} gave a description of the multiplicative structure of $\Omega^{SU}/ \textit{torsion}$. These works gave an almost complete picture of the coefficient ring of $\MSU$.\\[1.000ex]
The most crucial part of Conner-Floyd's work was \cite{SU_Bordism} the construction of the group of spherical classes $\mathcal{W}_*$, and a relation of that with $\Omega^{SU}$. One of the important technical tools for this identification was the following isomorphism \cite[Theorem 17.1]{SU_Bordism}
\begin{equation} \label{msuCFIntro}
 \MSU \wedge \Sigma^{-2}\Sigma^{\infty} \mathbb{C P}^{\infty} \xrightarrow{\cong} \MU   
\end{equation}
In this work, we take a small step in generalizing Conner-Floyd's approach to studying $\MSU$ in the realm of $\A^1$-homotopy theory. 
\subsection*{Some context in $\mathds{A}^1$-homotopy theory} The notion of oriented cohomologies, in other words, theories having Thom isomorphism for all vector bundles is quite well studied in $\A^1$-homotopy theory. Motivic cohomology ($\HZ$), and algebraic $K$-theory ($\text{KGL}$) are examples of such theories \cite{PaninOriented2}. Voevodsky's algebraic cobordism $\MGL$ is the universal oriented cohomology theory in the stable motivic homotopy category $\SH(S)$ \cite{PPR}.\\[1ex]
Panin-Walter \cite{PW} constructed special linear algebraic cobordism spectrum $\MSL$ and proved its universality among $\SL$ oriented theories (in other words, theories with Thom isomorphism for vector bundles with trivialized determinant bundle). There are interesting theories like Balmer's higher Witt theory, Hermitian $K$-theory, and Milnor-Witt motivic cohomology, which are $\text{SL}$-oriented but not oriented.\\[1.1ex]
In the realm of stable $\A^1$-homotopy theory, there are two elements $2 \in \pi_{0,0}(\mathds{1})$, and $(-\eta) \in \pi_{1,1}(\mathds{1})$, which under real realization go to $2$ \cite[Proposition 3.1.3]{etaRealRealization}. In general $2$ is not invertible in interesting theories like Witt groups. Ananyevskiy \cite{SLprojectiveBundleTheorem} showed projective bundle theorem for $\text{SL}$-oriented theories after inverting $\eta$. For more justification regarding inverting $\eta$ to better understand $\text{SL}$-oriented theories see \cite[Section 6]{SLprojectiveBundleTheorem}. Bachmann-Hopkins \cite[Theorem 8.8]{bachhopetaperiod} computed the homotopy groups of $\MSL[1/\eta]$ over fields. In our work, we generalize the isomorphism of \cref{msuCFIntro} in the stable $\A^1$-homotopy category.
\subsection*{What exactly we have done} By construction there is a ring spectrum map $\MSL \to \MGL$. For $\MGL$, we have a canonical orientation map $\Sigma^{-2,-1}\Sigma^{\infty} \mathds{P}^{\infty} \to \MGL$. Combining these two and the ring structure of $\MGL$, we get a map
\begin{equation} \label{main_question_intro}
    \Psi: \MSL \wedge \Sigma^{-2,-1}\Sigma^{\infty} \mathbb{P}^{\infty} \to \MGL
\end{equation}
Taking its complex realization, we recover the map in \cref{msuCFIntro}. We show the following.
\begin{theorem} \label{resultIntro} (\cref{onlyresultZ})
    Let $S$ be a Noetherian scheme of finite Krull dimension. Then the map \[\Psi: \MSL \wedge \Sigma^{-2,-1}\Sigma^{\infty} \mathbb{P}^{\infty} \to \MGL\] is an isomorphism in $\SH(S)$. 
\end{theorem}
The main idea of the argument is sketched below. 
\begin{itemize}
    \item First we show that the map $\Psi$ induces isomorphism in $E$ homology for some oriented motivic spectrum $E \in \SH(k)$ with additive formal group law. In other words, $\Psi \wedge E$ is an isomorphism in $\SH(k)$, for any perfect field $k$ (\cref{MSLsmashEequivalence}). In particular we get that $\Psi \wedge \HZ$ is an isomorphism.
    \item Let $\eta \in \pi_{1,1}(\mathds{1})$ be the first stable Hopf map in the (1,1)-stable homotopy group of the motivic sphere spectrum. We show that after inverting $\eta$, the map $\Psi[1/\eta]$ is an isomorphism ( \cref{LevineIdeaForGeneralBase}).\\[1.000000ex]
    The main idea here is that after inverting $\eta$, both, the domain and the codomain of the map $\Psi$ become contractible in $\SH(k)$, for any field $k$. So, $\Psi[1/\eta]$ is in fact an isomorphism of contractible spectra.
    \item We also show that when we work over a perfect field $k$, after $\eta$-completion, the map $\Psi^{\wedge}_{\eta}$ is an isomorphism in $\SH(k)$ (\cref{PsiIntegralActual}).\\[1.00000ex]
    As we have shown before that $\Psi \wedge \HZ$ is an isomorphism, $\Psi^{\wedge}_{\HZ}$ has to be an isomorphism as well. Here we compare completion with respect to $\HZ$, and $\eta$. With good enough conditions as mentioned in the proof of \cref{PsiIntegralActual}, $\HZ$- completion, and $\eta$-completion become equivalent. That is how we get that $\Psi^{\wedge}_{\eta}$ is an isomorphism.
    \item Once we have that $\Psi^{\wedge}_{\eta}$ and $\Psi[1/\eta]$ both are equivalences, we use the following standard pullback square to get that the map $\Psi$ is an equivalence in $\SH(k)$ (\cref{onlyresult}). Here $E \in \SH(k)$ is any motivic spectrum.
    \begin{equation*} 
\begin{tikzcd}[sep=small]
	E && {E^{\wedge}_\eta} \\
	\\
	{E [1/\eta]} && {E^{\wedge}_\eta [1/\eta]}
	\arrow[from=1-1, to=1-3]
	\arrow[from=1-1, to=3-1]
	\arrow[from=3-1, to=3-3]
	\arrow[from=1-3, to=3-3]
	\arrow["\ulcorner"{anchor=center, pos=0.125, rotate=180}, draw=none, from=3-3, to=1-1]
\end{tikzcd}
 \end{equation*}
 \item To generalize this from perfect fields to general base schemes, we use the fact that the map $\Psi$ is compatible with respect to pullback along change of base. It is enough to show that $\Psi$ is an equivalence in $\SH(\Z)$.\\
 As per \cref{detecting0inSHZ}, $\Psi$ is an equivalence in $\SH(\Z)$ iff $\Psi$ is an equivalence in $\SH(\Q)$ and $\SH(\Z/p\Z)$ for all prime $p$. As we have the equivalence over any perfect field, we are done.
\end{itemize}
\cref{resultIntro} gives us the following filtration
\begin{equation}\label{skeletalIntro}
 \begin{tikzcd}[column sep=tiny, row sep =small]
	{\MSL } && {\MSL \wedge \Sigma^{-2,-1} \Sigma^{\infty} \mathds{P}^2} && {\MSL \wedge \Sigma^{-2,-1} \Sigma^{\infty} \mathds{P}^3} & {...} & \MGL \\
	&& {\Sigma^{2,1} \MSL} && {\Sigma^{4,2} \MSL}
	\arrow[hook, from=1-1, to=1-3]
	\arrow[hook, from=1-3, to=1-5]
	\arrow[hook, from=1-5, to=1-6]
	\arrow[from=1-6, to=1-7]
	\arrow[from=1-3, to=2-3]
	\arrow[from=1-5, to=2-5]
\end{tikzcd}   
\end{equation}
Here all the projective spaces are considered to be internally pointed at $\infty$. 
\begin{itemize}
    \item From the resulting spectral sequence of the filtration \cref{skeletalIntro}, over a field of characteristic other than 2, we get (\cref{piZEROmsl}) 
    \begin{equation} \label{piZEROmslIntro}
 {\pi}_0(\MSL)/\eta \cong {\pi}_0(\MGL) \cong K^{M}_*(k)   
\end{equation}
\item Applying the zeroth slice functor $s_0$ on \cref{skeletalIntro}, we get 
\begin{corollary} (\cref{s0MSL}) \label{s0MSLIntro}
 Over any Noetherian, finite-dimensional scheme, \[s_0(\MSL) \cong s_0(\MGL) \cong s_0(\mathds{1}) \] 
\end{corollary}
\cref{piZEROmslIntro}, and \cref{s0MSLIntro} match with the already well-known computation regarding the unit map of $\MSL$ \cite[Example 16.35]{normsinmotivichomotopy}, or \cite[Proposition 3.6.3]{Yakerson_2021}. 
\item We also compute the first slice of $\MSL$. 
\begin{corollary} (\cref{s1MSLfinal})\label{s1MSLfinalIntro}
Over any Noetherian, finite-dimensional scheme, we have 
\[s_1(\MSL) \cong s_1(\mathds{1})\]
In particular, over a field of characteristic $0$, \[s_1(\MSL) \cong s_1(\mathds{1}) \cong \Sigma^{1,1}\HZ/2\]
\end{corollary}
\item Based on an idea of Oliver R\"ondigs, we also compute the first line of the stable homotopy group of $\MSL$ over a field of characteristic $0$. Let $\textbf{kq}$ denote the very effective cover of the Hermitian $K$-theory spectrum $\text{KQ}$.
\begin{theorem} (\cref{pi1msloliver}) \label{pi1msloliverintro}
Over a field of characteristic $0$, we have 
\[\pi_1(\MSL) \cong \pi_1(\textbf{kq})\]
    
\end{theorem}
\end{itemize}
\textbf{Outline of the paper}
For the sake of being self-contained, and notationally coherent, we have recalled well-known constructions and results in \cref{Notation and preliminaries} with due credits, which are used throughout \cref{An Interpolation between special linear and general algebraic cobordism Section}.\\[1.0000ex]
Some of the results, for example, \cref{cellularity MSL} regarding cellularity of $\MSL$ seem to be well-known, but we could not find a detailed written proof. We have spelled these out.\\[1.0ex]
We state the main question of the paper \cref{main_question} and prove \cref{resultIntro} in \cref{An Interpolation between special linear and general algebraic cobordism Section}. Some of its applications are shown in \cref{Some applications}.\\[1.000ex]
\textbf{Acknowledgement}\\[1.000ex]
I am thankful to my doctoral supervisor Oliver R\"ondigs for suggesting me the problem, teaching me most of the materials from scratch, sharing his ideas throughout this project, answering numerous questions, and lot of comments in several drafts of this article. I am extremely grateful to Matthias Wendt for going through an earlier version of this work, and correcting a lot of mistakes. I am thankful to Alexey Ananyevskiy for pointing out a mistake in \cref{first_computation}. This is part of the author's doctoral thesis. The doctoral position was funded by DFG-Sachbeihilfe \say{Algebraic bordism spectra: Computations, filtrations, applications}. The author was supported by the Nederlandse Organisatie voor Wetenschappelijk Onderzoek (Dutch Research Council) Vidi grant no VI.Vidi.203.004 during final stages of her work.

\section{Notations and Preliminaries} \label{Notation and preliminaries}
\subsection{Notations}
Let $S$ be a Noetherian scheme of finite Krull dimension. We are working in the stable motivic homotopy category of $S$, denoted as $\SH(S)$. We denote the category of smooth schemes of finite type over $S$ by $\SmS_S$. We start with the category of space (or Kan complexes) valued presheaves over $\SmS_S$, denoted as $\Preshv(\SmS_S)$. We enforce Nisnevich descent, and make the affine line invertible to get the unstable motivic homotopy category $\mathcal{H}(S)$. Passing to its pointed version, and making the operation $\wedge ({\mathds{P}^1})$ invertible, we get the stable motivic homotopy category over $S$.
\[\SH(S):= \mathcal{H}(S)_* [({\mathds{P}^1})^{-1}] \cong  \Bigl(L_{\text{Nis}}(\Preshv(\SmS_S)) \cap L_{{\A}^1}(\Preshv(\SmS_S))\Bigr)_*[({\mathds{P}^1})^{-1}]\]
In this paper, we use the following notation, unless specified otherwise.
\begin{center}
\begin{tabular}{ |c|c| } 
\hline
 $\mathds{1}$ & tensor unit of $\SH(S)$ or the motivic sphere spectrum\\
 $S^{1,0}$, $S^{1,1}$ & the simplicial circle, the Tate circle $\mathds{G}_m$\\
 $T:= \mathds{A}^1/(\mathds{A}^1-0)$ & $S^{1,0} \wedge S^{1,1} \cong \mathds{P}^1$\\
 $\SH^{\text{eff}}(S)$, $\SH^{\text{veff}}(S)$ & the effective, and the very effective subcategories of $\SH(S)$\\
 
 \hline
 \end{tabular}
\end{center}
For some spectrum $E \in \SH(S)$, and $X \in \SmS_S$, we denote 
 \begin{center}
\begin{tabular}{ |c|c| }
\hline
$\underline{\pi}_{p,q}(E)(-), \pi_{p,q}(E)$ & bigraded stable homotopy sheaves, and groups of $E$\\
$E^{p,q}(X)$ & $\Hom_{\SH(S)}(\Sigma^{\infty} X_{+}, \Sigma^{p,q} E)$\\
$E_{p,q}(X)$ & $\pi_{p,q} (\Sigma^{\infty} X_{+} \wedge E)$ \\
$f_q E$, $\tilde{f}_q E$ & $q$-th effective, and $q$-th very effective cover of $E$\\
$s_q E$, $\tilde{s}_q E$ & $q$-th effective, and $q$-th very effective slice of $E$\\

\hline

\end{tabular}
\end{center}

\begin{center}
\begin{tabular}{ |c|c| } 
\hline
 $\HZ$, $\text{H}\Z/p$, $\text{H}\Q$ & integral, mod-$p$, and rational motivic cohomology\\
 $\MGL$, $\MSL$, $\text{MSp}$ & general linear, special linear, and symplectic algebraic cobordism\\
 $\text{KQ}$, $\textbf{kq}$ & Hermitian $K$-theory, very effective cover of Hermitian $K$-theory\\
 $\underline{K}^{MW}$, $\underline{K}^{M}$ & Milnor-Witt $K$-theory, and Milnor $K$-theory sheaves\\
 \hline

\end{tabular}
\end{center}
We describe our main players algebraic cobordism $\MGL$, and special linear algebraic cobordism $\MSL$ as symmetric $\T$-spectra, and symmetric $\T^2$-spectra respectively. For more on (motivic) symmetric spectra, we would like to refer to \cite{Jardine2000}.

\subsection{General and special linear algebraic cobordism}
\subsubsection{Algebraic cobordism $\MGL$}\label{constructMGL}
In the late 90's, Voevodsky constructed the algebraic cobordism spectrum $\MGL$ in $\SH(S)$ \cite[Subsection 6.3]{voevodsky_icm}. If $S = Spec(k)$, where $k$ is a field with a complex embedding, complex realization of $\MGL$ gives complex cobordism spectrum $\MU$ \cite[p.166]{summermotivic}.\\[1.000000000ex]
Let $V$ be a vector bundle over a smooth scheme $X$ over $S$. The Thom space of $V$, $\text{Th}(V)$ is defined as the Nisnevich sheafification of the presheaf $(V/ V - z(X))$. Here $z:X \to V$ is the zero section. Sometimes it is also denoted by $\text{Th}_X(V)$. \\[1.0ex]
Let's denote the grassmannian of $n$-planes $\Gr(n, (\mathcal{O}_{S}^{n})^p)$ as $\Gr(n,np)$. Let $\mathcal{T}_{n,np}$ be the tautological $n$-bundle on it.\\
The closed immersion $\Gr(n,np) \to \Gr(n,n(p+1))$ is covered by the vector bundle map $\mathcal{T}_{n,np} \to \mathcal{T}_{n,n(p+1)}$, resulting into maps between the
Thom spaces $\text{Th}(\mathcal{T}_{n,np}) \to \text{Th}(\mathcal{T}_{n,n(p+1)})$.
We denote
\begin{align*}
     \mathcal{T}_n & := \ncolim_{p} \mathcal{T}_{n,np}\\
     \BGL_n & := \ncolim_{p} \Gr_{n,np}\\
     \MGL_n & := \ncolim_{p} \text{Th}(\mathcal{T}_{n,np})
\end{align*}
The structure map of $\MGL$ is induced by the isomorphism
\[\mathcal{T}_{n+1}\restriction_{\BGL_n} \cong \mathcal{T}_{n} \oplus \mathcal{O}_{\BGL_n}\]
This gives the map
\[T \wedge \text{Th}_{\BGL_n}(\mathcal{T}_n) \cong \text{Th}_{\BGL_n}(\mathcal{O}_{\BGL_n}) \wedge \text{Th}_{\BGL_n}(\mathcal{T}_n) \cong \text{Th}_{\BGL_n}(\mathcal{O} \oplus \mathcal{T}_{n}) \to \text{Th}_{\BGL_{n+1}}(\mathcal{T}_{n+1}) \]
\subsubsection*{Symmetric $T$-spectrum structure}
In \cite{vezzosi_BP}, \cite{PPR}, $\MGL$ is given ring structure. We now recall the \textbf{symmetric $\T$-spectrum} structure of $\MGL$ Following \cite[section 4]{PW}.\\[1.000ex]
The inclusion $\Gr(n,np) \to \Gr(n,n(p+1))$ is compatible with the diagonal action of $\GL_n := \GL(\mathcal{O}_S^n)$ on the grassmannians $\Gr(n,np)$. The bundles $\mathcal{T}_{n,np}$ are also $\GL_n$ equivariant.
 The inclusion $\Gr(n,np) \to \Gr(n,n(p+1))$ induces a $\GL_n$-equivariant inclusion of the Thom spaces $\text{Th}(\mathcal{T}_{n,np}) \to \text{Th}(\mathcal{T}_{n,n(p+1)})$. This gives an $\GL_n$-action on the space $\MGL_n$. This induces an action of $\Sigma_n < \GL_n$ on $\MGL_n$.\\[1.0000ex]
We have a closed embedding 
\[\Gr(n,np) \times \Gr(m,mp) \to \Gr(n+m,(n+m)p)\]
This is covered by the map of tautological vector bundles, inducing the following map on Thom spaces  
\begin{equation}\label{ringStrMGL1}
 \text{Th}(\mathcal{T}_{n,np}) \wedge  \text{Th}(\mathcal{T}_{m,mp}) \to \text{Th}(\mathcal{T}_{n+m,(n+m)p})
\end{equation}
In the colimit it induces the following map
\begin{equation}\label{ringStrMGL2}
 \mu_{nm} : \MGL_n \wedge \MGL_m \to \MGL_{n+m}
\end{equation}
As the map \cref{ringStrMGL1} is $\GL_n \times \GL_m$-equivariant, so is the map \cref{ringStrMGL2}.\\[1.000ex]
The inclusion $(\text{id},0,0,..0):\mathcal{O}_S^n \to (\mathcal{O}_S^n)^{\oplus p}$ induces an inclusion $x_n: \Gr(n,n) \to \Gr(n,np)$. Passing to the colimit, this makes $\BGL_n$ a pointed space. This $x_n$ in $\BGL_n$ is an $S$-valued point, which is fixed under the $\GL_n$-action. The fiber of $\mathcal{T}_{n}$ over this point is $\mathcal{O}_S^n$. This gives us the required unit map 
\[e_n: \T^n \xrightarrow{\cong} \text{Th}(\mathcal{O}_S^n) \to \MGL_n\]
\subsubsection{Special linear algebraic cobordism $\MSL$}\label{MSL}
Panin and Walter constructed the special linear algebraic cobordism spectrum $\MSL$ \cite{PW}. We recall the construction following \cite[section 4]{PW}.\\[1.0000ex]
Let $n >0$. Consider the determinant line bundle $det(\mathcal{T}_{n,np})$ on $\Gr(n,np)$. The \textbf{oriented grassmannian} $\SGr_{n,np} \in \SmS_S$ is defined as the compliment of the zero section of this determinant bundle. Let $z:\Gr_{n,np} \to  det(\mathcal{T}_{n,np})$ be the zero section.
\[\SGr_{n,np}:= det(\mathcal{T}_{n,np}) \backslash z (\Gr_{n,np})\]
We have the following pullback diagram,\\
\begin{equation} \label{MSLconstruct}
\begin{tikzcd}
	{\pi^*(\mathcal{T}_{n,np})} & {\mathcal{T}_{n,np}} \\
	{\SGr_{n,np}} & {\Gr_{n,np}}
	\arrow["\pi"', from=2-1, to=2-2]
	\arrow["{\cong_{\mathds{A}^1}}", from=1-2, to=2-2]
	\arrow[from=1-1, to=1-2]
	\arrow[from=1-1, to=2-1]
\end{tikzcd}
\end{equation}
We denote $\pi^*(\mathcal{T}_{n,np})$ by $\widetilde{\mathcal{T}_{n,np}}$.
We define
\begin{align*}
     \BSL_n & := \ncolim_{p} \SGr_{n,np}\\
     \MSL_n & := \ncolim_{p} \text{Th}(\widetilde{\mathcal{T}_{n,np}})
\end{align*}
In the following commutative diagram \cref{MSLstrMap}, the dotted arrow is induced by the map $\mathcal{T}_{n}\oplus \mathcal{O} \to \mathcal{T}_{n+1}$, covering the map $\BGL_{n} \xrightarrow{i_n} \BGL_{n+1}$. Here we need to use the fact that $det(\mathcal{T}_{n} \oplus \mathcal{O}) \cong det(\mathcal{T}_{n})$. 
\begin{equation}\label{MSLstrMap}
\begin{tikzcd}[sep=scriptsize]
	{\widetilde{\mathcal{T}_{n}} \oplus \mathcal{O}} && {\mathcal{T}_{n}\oplus \mathcal{O}} & {} & {\mathcal{T}_{n+1}} \\
	&&& {\widetilde{\mathcal{T}_{n+1}}} & {} \\
	{det(\mathcal{T}_{n} \oplus \mathcal{O})\backslash z(\BGL_{n})} && {\BGL_{n}} && {\BGL_{n+1}} \\
	{\BSL_n:=det(\mathcal{T}_{n})\backslash z(\BGL_{n})} &&& {det(\mathcal{T}_{n+1})\backslash z(\BGL_{n+1}) =: \BSL_{n+1}} \\
	{}
	\arrow[from=1-1, to=1-3]
	\arrow[from=1-1, to=3-1]
	\arrow[from=1-3, to=1-5]
	\arrow[from=1-3, to=3-3]
	\arrow[from=1-5, to=3-5]
	\arrow[from=2-4, to=1-5]
	\arrow[from=2-4, to=4-4]
	\arrow[from=3-1, to=3-3]
	\arrow["\cong", from=3-1, to=4-1]
	\arrow["{i_n}"{pos=0.6}, hook, from=3-3, to=3-5]
	\arrow["{\pi_n}"'{pos=0.8}, from=4-1, to=3-3]
	\arrow["\widetilde{i_n}",dashed, from=4-1, to=4-4]
	\arrow["{\pi_{n+1}}"'{pos=0.7}, from=4-4, to=3-5]
\end{tikzcd}
\end{equation}
As $(\pi_{n+1} \circ \widetilde{i_n})^* \mathcal{T}_{n+1} \cong (i_n \circ \pi_n)^* \mathcal{T}_{n+1}$, we get that $\widetilde{\mathcal{T}_{n+1}}\restriction_{\BSL_n} \cong \widetilde{\mathcal{T}_{n}} \oplus \mathcal{O}$. This induces the structure map of $\MSL$.
\subsubsection*{Symmetric $T^2$-spectrum structure}
The map $\SGr_{n,np} \to \Gr_{n,np}$ is a principal $\Gm$-bundle. We have the following pullback diagram. 
\[\begin{tikzcd}[sep=scriptsize]
	{det(\widetilde{\mathcal{T}_{n,np}}) \cong \pi^*(det(\mathcal{T}_{n,np}))} && {det(\mathcal{T}_{n,np})} \\
	{\SGr_{n,np}} && {\Gr_{n,np}}
	\arrow[from=1-1, to=1-3]
	\arrow[from=1-1, to=2-1]
	\arrow[from=1-3, to=2-3]
	\arrow["\pi"', from=2-1, to=2-3]
\end{tikzcd}\]
The inclusion $\SGr_{n,np} \to det(\mathcal{T}_{n,np})$ gives a nowhere vanishing section of $det(\widetilde{\mathcal{T}_{n,np}})$. The corresponding trivialization is denoted by $\lambda_{p,np}: \mathcal{O}_{\SGr_{n,np}} \xrightarrow{\cong}  det(\mathcal{T}_{n,np})$.\\[1.0000ex]
There is a bijection between smooth scheme maps $X \to \SGr_{n,np}$ and pair of maps $(f,\tau)$, where $f: X \to \Gr_{n,np}$ and $\tau: \mathcal{O}_X \xrightarrow{\cong} f^* det(\mathcal{T}_{n,np})$. Using this we get the following maps of representable presheaves\\
\begin{equation*}
\begin{split}
    \Hom_{Sm_S} (X, \SGr_{n,np}) \times \Hom_{Sm_S} (X, \SGr_{m,mp}) \to \Hom_{Sm_S} (X, \SGr_{(m+n),(m+n)p})\\
    (f_1, \tau_1) \times (f_2, \tau_2) \to (f_1 \oplus f_2, \tau_1 \oplus \tau_2)
    \end{split}
\end{equation*}
In the colimit it induces the required multiplication map:
\begin{equation*}
\begin{split}
    \BSL_n \times \BSL_m \to \BSL_{n+m}\\
    \MSL_n \wedge \MSL_m \to \MSL_{n+m}
    \end{split}
\end{equation*}
As the bundle $\mathcal{T}_{n,np}$ on $\Gr_{n,np}$ is $\GL_n$ equivariant, so is the determinant bundle $det(\mathcal{T}_{n,np})$. This also induces $\GL_n$-action on the complement of its zero section $\SGr_{n,np}$. Taking the colimit, it gives the required action on $\MSL_n$.\\[1.000ex]
There is a subtle problem in constructing the unit map. To construct the unit maps $e_n: T^n \to \MGL_n$, we use the $S$-valued fixed point $\Gr(n,n) \to \BGL_n$. In order to construct similar unit maps of $\MSL$, we would need $\Sigma_n$-fixed points on $\BSL_n$, preferably lying over $\Gr(n,n)$. As $\SGr_{n,np} \to \Gr_{n,np}$ is a principal $\mathbb{G}_m$-bundle, the $\mathbb{G}_m$ fiber over $\Gr(n,n)$ is not fixed by $\Sigma_n \subseteq \GL_n$. But it is fixed by the alternating group $\mathfrak{U}_n \subset \SL_n$. Therefore the embedding $\Sigma_n \to \SL_{2n}$ is used.\\ It gives an action $\Sigma_n \times \BSL_{2n} \to \BSL_{2n}$, preserving the pointwise fiber over the distinguished base point lying over $\Gr(2n,2n) \to \BGL_{2n}$. It is more convenient to write $\MSL$ as a symmetric $T^2$-spectra. Each space of the commutative $\T^2$-monoid $\MSL$ is $\MSL_{2n}$. We skip the subtle technicality of this construction, and refer to \cite[Section 4]{PW}.
\begin{remark}\label{complexrealization}
Let's denote the algebraic subgroup of rank $n$ matrices  of $M_{n \times N}$, by $U(n,N)$. From \cite[Remark 4.2.5]{MariiaYakersonThesis}, we can describe the Grassmannian $\SGr_{n,N}$ as the quotient sheaf $\SL_n\backslash U(n,N)$.\\[1.000ex]
As the complex realization functor is a left adjoint \cite[22.3.1]{HandbookUnstable}, it commutes with colimits. If we start from $\SH(\C)$, taking complex realization of $\BSL_n$ gives $\text{BSU}_n$. Comparing with the construction of $\text{MSU}$ \cite[p.13]{Limonchenko_2019}, complex realization of $\MSL$ gives the spectrum $\MSU$.

\end{remark}
\begin{remark}\label{Remark3.1}
  By the geometric construction, there is a canonical map $\Phi : \MSL \to \MGL$ induced by the map $\pi^*(\Taut\GL_{n,np}) \to \Taut\GL_{n,np}$ in diagram \ref{MSLconstruct}. If $\MGL$ is considered as a symmetric $T^2$-spectrum, $\Phi$ induces a symmetric $T^2$-ring spectrum map.
\end{remark}
\subsection{Stable Cellularity}
The classical stable homotopy category $\SH$ is a compactly generated triangulated category with a single compact generator. The topological sphere spectrum $\Sp$ is a compact generator of this category.\\[1.000ex]
The motivic stable homotopy category $\SH(S)$ does not have single compact generator. From \cite[Proposition 5.3.3]{robaloThesis}, $\{\Sigma^{\infty} X_{+} | X \in \SmS_S\}$ is a set of compact generators for $\SH(S)$. We can consider the subcategory generated by the motivic sphere spectrum $\Sigma^{\infty} S_+$. This is called the subcategory of \textbf{cellular spectra}.
\begin{definition}\label{DefinitionStableCellularity}\cite[Definition 2.10]{cellulardi} \cite[Remark 4.1]{motlandweber}
   The subcategory of \textbf{cellular objects} is the smallest localizing triangulated subcategory of $\SH(S)$ containing all the motivic spheres $\{\Sigma^{\infty}_{\mathds{P}^1}S^{p,q}| p,q \in \Z\}$. This subcategory is closed under homotopy colimits.
\end{definition}
We call an object $X \in \mathcal{H}(S)_*$ to be (stably) cellular if $\Sigma^{\infty}_{\mathds{P}^1} X_+$ is a cellular object in $\SH(S)$.\\
As in this article, we are not going to talk about unstable cellular objects, by cellularity we will always mean stable cellularity. For unstable cellularity, or in general, the notion of cellularity in a pointed model category, see \cite{cellulardi}.
\begin{example}
\cite[Lemma 4.1]{voevodsky_icm}    As $\mathds{A}^n \backslash 0 \cong_{\mathds{A}^1} S^{2n-1,n}$, $\mathds{A}^n \backslash 0$ is cellular.
\end{example}
\begin{lemma} \label{lemma 2 out of 3} \cite[Lemma 3.11]{cellulardi} 
 Let $A \to B \to C $ be a cofiber sequence in $\SH(S)$. If any two of them are cellular, so is the third. \end{lemma}
\begin{example} 
 $\mathds{P}^1 \cong S^{2,1}$. An induction argument using the cofiber sequence $\mathds{A}^n - 0 \to \mathds{P}^{n-1} \to \mathds{P}^{n}$ and \cref{lemma 2 out of 3} shows $\mathds{P}^{n}$ is cellular.
\end{example}
\begin{example}\cite[Lemma 3.11]{cellulardi}
    let's consider a smooth scheme $X$, and a closed point $x \in X$. We can show that $X$ is cellular iff $X \backslash x$ is.\\[1.01ex]
    Let's denote the normal bundle of $x \hookrightarrow X$ as $\mathcal{N}_x$. We have the following cofiber sequence using homotopy purity \cite[p.~115, Theorem 2.23]{MV99}:
    \[X \backslash x \hookrightarrow X \to \text{Th}(\mathcal{N}_x)\]
 We know that $\mathcal{N}_x \cong \A^n/(\A^n \backslash 0) \cong S^{2n,n}$, where $n$ is the dimension of $X$. Using \cref{lemma 2 out of 3}, we are done.  
\end{example}
\begin{remark}
    For any closed subscheme $Z \hookrightarrow X$, we have the cofiber sequence \[Z \hookrightarrow X \to \text{Th}(\mathcal{N}_Z)\]
    As mentioned in \cite[Remark 3.12]{cellulardi}, due to a  lack of description of $\text{Th}(\mathcal{N}_Z)$, we can not say that $X$ is cellular iff $X \backslash Z$ is. 
\end{remark}
\begin{example}
    The Grassmannians $\GL_n$, and algebraic cobordism spectrum $\MGL$ are cellular \cite[Proposition 4.4, Theorem 6.4]{cellulardi}. Cellularity of Grassmannians can be proved in a beautiful way using the description of the top dimensional Schubert cells \cite{wendt2012examples}.
\end{example}
Cellularity of $\MSL$ seems to be well-known. As we could not find a detailed proof in the literature, we proceed to write it down.
\begin{lemma} \label{cellularity MSL}
$\MSL$ is cellular.
\end{lemma}
\begin{proof}
Using \cite[lemma 6.1]{cellulardi} $\MSL$ is weakly equivalent to \[\nhocolim ( \Sigma^{\infty} \MSL_0 \to \Sigma^{-2,-1}\Sigma^{\infty} \MSL_1 \to \Sigma^{-4,-2}\Sigma^{\infty} \MSL_2 \to ....)\] It is enough to show $\MSL_n$ is cellular.\\[1.0000ex]
Exactly the same argument works if we want to prove that $\MSL$ as a $\T^2$-symmetric spectrum is cellular in $\SH(S)$, considered as the homotopy category of $\T^2$-spectra \cite[Theorem 3.2]{PW}. In that case, we consider $\MSL$ as the homotopy colimit of 
\[\Sigma^{\infty} \MSL_0 \to \Sigma^{-4,-2}\Sigma^{\infty} \MSL_2 \to...\to \Sigma^{-4n,-2n}\Sigma^{\infty} \MSL_{2n} \to ....\]
It is then enough to show that $\MSL_{2n}$ is cellular.\\[1.0000ex]
Using the notations from \cref{MSL},
\[\MSL_n := \ncolim_p \text{Th} (\widetilde{\mathcal{T}_{n,np}}|_{\SGr_{n,np}})\]
Here $\widetilde{\mathcal{T}_{n,np}}$ is the pullback of the tautological bundle $\mathcal{T}_{n,np}$ on the Grassmannian $\Gr_{n,np}$, along the map $\SGr_{n,np} \to \Gr_{n,np}$.\\[1.0000ex] 
Let $z: \Gr_{n,np} \to det(\mathcal{T}_{n,np})$ be the zero section, and $\mathcal{N}_z$ be its normal bundle. We have the following commutative diagram:\\
\begin{equation} \label{MSLdiagramCellularity}
    \begin{tikzcd}
	{\pi'^*r^*(\mathcal{T}_{n,np})} & {r^*(\mathcal{T}_{n,np})} & {\mathcal{T}_{n,np}} \\
	\\
	{\SGr_{n,np} := det(\mathcal{T}_{n,np})\backslash z(\Gr_{n,np})} & {det(\mathcal{T}_{n,np})} & {\Gr_{n,np}}
	\arrow["q","{\cong_{\mathds{A}^1}}"', from=1-3, to=3-3]
	\arrow["r", "{\cong_{\mathds{A}^1}}"', from=3-2, to=3-3]
        \arrow["z", bend left, from=3-3, to=3-2]
	\arrow[from=1-2, to=1-3]
	\arrow["p", "{\cong_{\mathds{A}^1}}"', from=1-2, to=3-2]
	\arrow["{\pi'}"', from=3-1, to=3-2]
	\arrow[from=1-1, to=3-1]
	\arrow[from=1-1, to=1-2]
\end{tikzcd}
\end{equation}
Using \cite[lemma 3.5]{slices_MGL}, we get
\[\cofib (\text{Th} (\pi'^*r^*(\mathcal{T}_{n,np})) \to \text{Th} (r^*(\mathcal{T}_{n,np}))) \cong_{\mathds{A}^1} \text{Th}\Bigl(z^*\bigl(r^*(\mathcal{T}_{n,np})\bigr) \oplus \mathcal{N}_z \Bigr)\]
We have 
\[\text{Th}(z^*r^*(\mathcal{T}_{n,np}) \oplus \mathcal{N}_z) \cong \text{Th}((r \circ z)^*(\mathcal{T}_{n,np}) \oplus \mathcal{N}_z)\]

 Using $r \circ z = id_{\Gr_{n,np}}$, and $\text{Th}(\mathcal{N}_z )\cong \text{Th} (det(\mathcal{T}_{n,np})_{\Gr_{n,np}})$, we get the following.
 \[\text{Th}(z^*r^*(\mathcal{T}_{n,np}) \oplus \mathcal{N}_z) \cong \text{Th} ((\mathcal{T}_{n,np})_{\Gr_{n,np}}) \wedge \text{Th} (det(\mathcal{T}_{n,np})_{\Gr_{n,np}})\]
From the proof of \cite[theorem 6.4]{cellulardi}, $\text{Th} ((\mathcal{T}_{n,np})_{\Gr_{n,np}})$ is cellular.
Using \cite[Lemma 4.15]{bachhopetaperiod}, $\SGr_{n,np}$ is cellular.\\[1.000ex]
There is a cofiber sequence \begin{equation} \label{cofiber_SGr}
    \SGr_{n,np} \to \Gr_{n,np} \to \text{Th} (det(\mathcal{T}_{n,np})_{\Gr_{n,np}})
\end{equation} 
Using \cref{cofiber_SGr} and \cref{lemma 2 out of 3}, $\text{Th}(det(\mathcal{T}_{n,np})_{\Gr_{n,np}})$ is cellular.\\[1.00ex]
Now, consider the cofiber sequence
\[\text{Th} (\pi'^*r^*(\mathcal{T}_{n,np})) \to \text{Th} (r^*(\mathcal{T}_{n,np})) \to \text{Th} ((\mathcal{T}_{n,np})_{\Gr_{n,np}}) \wedge \text{Th} (det(\mathcal{T}_{n,np})_{\Gr_{n,np}})\]
We have shown that the cofiber is cellular, as cellular objects are closed under smash product \cite[Lemma 3.3]{cellulardi}.\\[1.000ex] 
As vector bundle projection is $\A^1$-weak equivance, we get  \cref{MSLdiagramCellularity}, \[\text{Th} (r^*(\mathcal{T}_{n,np})) \cong \text{Th} (\mathcal{T}_{n,np})\] So, it is also cellular.\\[1.000ex] 
That gives the cellularity of 
\[\MSL_n := \nhocolim_p \text{Th} (\pi'^*r^*(\mathcal{T}_{n,np}))\]
\end{proof}
Now, we are going to discuss more about how we can use strong tools from classical homotopy theory in stable motivic homotopy category, when we restrict ourselves in the cellular subcategory.\\[1.01ex] 
In the category $\SH(S)$, the stable homotopy sheaves $\underline{\pi}_{p,q}(-)$ detect stable equivalences \cite[Lemma 3.7]{Jardine2000}. For any $E \in \SH(S)$, we can consider the stable homotopy groups $\pi_{p,q}(E) := \Hom_{\SH(S)}(\Sigma^{p,q} \mathds{1}, E)$. These groups are the global sections of the stable homotopy sheaves \[\underline{\pi}_{p,q}(E)(S) \cong \pi_{p,q}(E)\]
For a general spectrum $E \in \SH(S)$, $\pi_{p,q}(E) \cong 0$ for all integers $p,q$; does not imply that $E \cong 0$ in $ \SH(S)$. This is expected, as we need to consider maps from all generators of $\SH(S)$, not just $\mathds{1}$. Also, stable equivalences between generic motivic spectra cannot be detected by these groups $\pi_{p,q}(-)$.\\[1.01ex]
For cellular spectra, it is enough to consider maps from the sphere spectrum. 
\begin{lemma}\label{CellularHomotopyGroupsDetects} \cite[Proposition 7.1]{cellulardi}
 Let $E \in \SH(S)$ be a cellular spectrum. Then $\pi_{p,q}(E) \cong 0$ for all integers $p,q$; iff $E \cong 0$ in $ \SH(S)$. 
\end{lemma}
As a corollary of \cref{CellularHomotopyGroupsDetects}, we get that a map between cellular spectra is an equivalence iff it induces isomorphism of the stable homotopy groups.
\begin{definition}
    For any spectrum $E \in \SH(S)$, and some $M \in \SH(S)$, we can talk about the $E$-cohomology and $E$-homology groups of $M$. 
    \begin{itemize}
        \item \textbf{$E-$cohomology group of $M$}: $E^{p,q}(M) := \Hom_{\SH(S)}(M, \Sigma^{p,q} E)$
        \item \textbf{$E-$homology group of $M$}: $E_{p,q}(M) := \pi_{p,q} (M \wedge E)$
    \end{itemize}
     As a special case, for $X \in \SmS_S$, we define 
    \begin{itemize}
        \item \textbf{$E-$cohomology} $E^{p,q}(X) := \Hom_{\SH(S)}(\Sigma^{\infty} X_{+}, \Sigma^{p,q} E)$
        \item \textbf{$E-$homology} $E_{p,q}(X) := \pi_{p,q} (\Sigma^{\infty} X_{+} \wedge E)$
    \end{itemize}
    
\end{definition}
To formulate something similar to K\"unneth spectral sequence, more carefully curated finiteness conditions are required. In \cite[Section 8]{cellulardi}, the smallest full triangulated subcategory containing all the spheres $S^{p,q}$ is called the class of \text{finite cell complexes}. It has all the finite direct sums, but might not have the infinite ones. We have the following analog of K\"unneth spectral sequence.
\begin{theorem} \label{TorSpectral} \cite[Theorem 8.6]{cellulardi}
 Let $E$ be a ring spectrum. Let $A$ and $B$ are motivic spectra where $A$ is a finite cell complex. There exists a strongly convergent trigraded spectral sequence\[[Tor_a^{E^{*,*}}(E^{*,*}A, E^{*,*}B)]^{(b,c)} \Longrightarrow E^{b-a,c}(A \wedge B)\]
\end{theorem}
If $E^{*,*}(A)$ is flat over $E^{*,*}$, due to the vanishing of higher Tor's, we get K\"unneth isomorphism 
\begin{equation} \label{Kunneth}
 E^{*,*}(A) \otimes_{E^{*,*}} E^{*,*}(B) \cong E^{*,*}(A \wedge B)   
\end{equation}
If we start with an oriented ring spectrum $E$ (see \cref{orientation}), this can be generalized a bit more. The class of \text{finite cell complexes} can be expanded to be closed under taking Thom spaces of vector bundles over smooth schemes. This class is defined as \text{linear motivic spectra} in \cite[Definition 8.9]{cellulardi}. For an oriented $E$, \cref{TorSpectral} applies when $A$ is linear \cite[Theorem 8.12]{cellulardi}. Using \cite[Remark 8.10]{cellulardi}, one can show that $\MGL$, and $\MSL$ are linear.

For a strongly dualizable cellular spectrum $A$ with $E^{*,*}(A)$, a free $E^{*,*}$ module of finite rank, and a cellular spectrum $B$, denoting the dual of $A$ as $A^{\vee}$ we have the following isomorphism \cite[p.22, Section 6.1]{Hoyoismalgmot} \begin{equation} \label{kunnethweuse}
   E_{*,*}(A^{\vee}) \otimes_{E_{*,*}} E_{*,*}(B) \cong E_{*,*}(A^{\vee} \wedge B) 
\end{equation}
The $E$ homology situation is conveniently generalized and stated for $E$ to be Eilenberg-Maclane spectra in \cite[Lemma 5.2]{Hoyoismalgmot}. Very convenient 
\say{finiteness} and cellularity assumptions can be stated for strongly dualizable spectra to make sure their $E$-cohomology and homology are dual to each other as specified in the following \cref{dualmsl}, formulated and proved in \cite[Corollary 4.10]{bachhopetaperiod}
\begin{proposition} \label{dualmsl}
 Let $X$ be the sequential colimit of the directed system $X_0 \to X_1 \to...$ in $\SH(S)$. Let $E$ be a (homotopy) commutative ring spectrum with $\lim_i^1 E^{*,*}(X_i) =0$. Then $E^{*,*}(X) \cong \lim_i (E^{*,*}(X_i))$. If each $X_i$ is cellular, strongly dualizable with $E^{*,*}(X_i)$ or $E_{*,*}(X_i)$ flat, we have that $E$-homology and cohomology of $X$ are dual to each other in the following sense \[E^{*,*}(X) \cong \Hom_{E_{*,*}}(E_{*,*}(X), E_{*,*})\] and \[E_{*,*}(X) \cong \Hom_{E_{*,*},c}(E^{*,*}(X), E_{*,*})\]
\end{proposition}
Here while dualizing, we consider only the maps from $\lim_p E^{*,*} (X_i)$ to $E_{*,*}$, which are continuous with respect to inverse limit topology on the domain (constructed taking discrete topology on each $E^{*,*}  (X_i)$), and discrete topology on the codomain. That effectively means only the maps factoring through $E^{*,*}  (X_i)$, for some finite $i$, are being considered.
\subsection{Homotopy \textit{t}-structure and completion}
\begin{definition} \label{definitionHomotopyTstructure}
       Let $\SH(S)_{\geq d}$ denote the subcategory of $\SH(S)$, generated by $\Sigma^{p,q} \Sigma^{\infty} X_{+}$ under homotopy colimits and extensions for $X \in Sm_S$, and $p-q \geq d$. We have that $\SH(S)_{\geq d}$ is the non-negative part of a unique $t$-structure using \cite[1.4.4.11]{HA}. It is called the homotopy $t$-structure.\\[1.000ex] From the canonical truncation functors, we get the following cofiber sequence for any spectrum $E$:
       \[E_{\geq d} \to E \to E_{\leq d-1} \to \Sigma^{1,0} E_{\geq d}\] 
   \end{definition}
   The functor $\tau_{\leq 0} \circ \tau_{\geq 0} \xrightarrow{\cong} \tau_{\geq 0} \circ \tau_{\leq 0} : \SH(S) \to \SH(S)^{\heartsuit}$ is denoted by $\underline{\pi}_0$. 
\begin{remark}
    A spectrum $E\in \SH(S)$ is $n$-connected if the homotopy sheaves ${\underline{\pi}}_{p,q}(E)$ vanish whenever $p-q < n$. Let's fix a base scheme $S$. We say \textbf{stable $\A^1$-connectivity property} holds over $S$, if the $\A^1$-localization functor $L_{\A^1}$ preserves $0$-connected or connected spectra. 
As Morel's stable $\A^1$-connectivity property holds over fields $k$ \cite{StableAoneConnectivity}, the previous formulation of the $t$-structure coincides with Morel's actual formulation of homotopy $t$-structure using homotopy sheaves $\underline{\pi}_{p,q}$ \cite[Theorem 2.3]{Hoyoismalgmot}.\\
For any spectrum $E\in \SH(k)$,
\begin{itemize}
    \item $E \in E_{\geq d}$ iff $\underline{\pi}_{p,q} = 0$ for $p-q < d$.
    \item $E \in E_{\leq d}$ iff $\underline{\pi}_{p,q} = 0$ for $p-q > d$.
\end{itemize}
\end{remark}
\begin{remark}
 The homotopy $t$-structure in $\SH(k)$ is left complete. In other words, the map $E \to \nholim_{n} E_{\leq n}$ is an equivalence \cite[Corollary 2.4]{Hoyoismalgmot}.
\end{remark}
Let $\mathcal{C}$ be any presentable, symmetric monoidal stable $\infty$-category. For any commutative algebra object $E \in \textit{CAlg}(\mathcal{C})$, the \textbf{$E$-nilpotent completion} of any $X \in \mathcal{C}$, denoted as $X^{\wedge}_{E}$, is defined as the limit of the standard cobar construction
\[\Delta_{+} \to \mathcal{C}\]  
\[[n] \to X \otimes E^{\otimes n+1} \]
    \[X^{\wedge}_{E} := \nholim_n (X \otimes E^{\otimes n+1})\]
For the rest of this section, we consider $(\mathcal{C},\otimes, \mathds{1})$ to be any presentable, symmetric monoidal stable $\infty$-category with a compatible $t$-structure (i.e. $\mathcal{C}_{\geq 0} \otimes \mathcal{C}_{\geq 0} \subseteq \mathcal{C}_{\geq 0}$).
\begin{theorem}\cite[Theorem 2.1]{Bachmann2021TopologicalMF} \cite[Lemma 7.2.12]{mantovani} \label{completion}
Let $\mathcal{C}$ be any presentable, symmetric monoidal stable $\infty$-category with a compatible $t$-structure, and the $t$-structure is left complete.
Consider a bounded below object $X \in \mathcal{C}$ ($X \in \mathcal{C}_{\geq n}$, for some $n \in \Z$), and some $E \in \textit{CAlg}(\mathcal{C}_{\geq 0})$. Additionally let's assume that $\underline{\pi}_0(E) \in \textit{CAlg}(\mathcal{C}^{\heartsuit})$ is an idempotent, that is $\underline{\pi}_0(E) \otimes^{\heartsuit} \underline{\pi}_0(E) \to \underline{\pi}_0(E)$ is an isomorphism.
Then there is an equivalence \[X^{\wedge}_E \to X^{\wedge}_{\underline{\pi}_0(E)}\]   
\end{theorem}
\begin{remark}\label{constructIdempotents}
 One standard way of constructing idempotent algebras in $\textit{CAlg}(\mathcal{C}^{\heartsuit})$ is the following. Let's consider $L_1,L_2,...,L_n \in \mathcal{C}_{\geq 0}$, and given maps $x_i: L_i \to \mathds{1}$. We define
\[ X/({x_1}^{m_1}, {x_2}^{m_2},...,{x_n}^{m_n}) := X \otimes \hocofib({x_1}^{m_1}: {L_1}^{m_1} \to \mathds{1}) \otimes... \otimes \hocofib({x_n}^{m_n}: {L_n}^{m_n} \to \mathds{1})\]
Here $x_i^{m_i}$ denotes tensor product, not iterated composition.\\
The object $\underline{\pi}_0(X/({x_1}^{m_1}, {x_2}^{m_2},...,{x_n}^{m_n})) \in \textit{CAlg}(\mathcal{C}^{\heartsuit})$ is an idempotent. This fact is mentioned and used in \cite[Section 2]{Bachmann2021TopologicalMF}. A proof can be found in \cite[Remark 7.2.4]{mantovani}.   
\end{remark}
We will use the following weaker version of \cite[Theorem 2.2]{Bachmann2021TopologicalMF}.
\begin{theorem}\label{completion2}
    Consider any bounded below $X \in \mathcal{C}$, and any $f: \mathds{1} \to \mathds{1}$. If $\mathcal{C}$ is left complete, there is an equivalence
    \[X^{\wedge}_{\underline{\pi_0}(\mathds{1}/f)} \xrightarrow{\cong} X^{\wedge}_f \]
\end{theorem}
\begin{remark}\label{constructEta}
   Let's consider the canonical map 
\[S^{3,2} \cong \A^2 \backslash \{0\} \xrightarrow{\tilde{\eta}} \mathds{P}^1 \cong S^{2,1} \]
\[(a,b) \to [a:b]\]
As it is a map between spheres, after stabilizing, we get a map of the sphere spectrum
\[\eta : \Sigma^{1,1} \mathds{1} \to \mathds{1}\]
This map $\eta$ is called the first stable Hopf map. Taking real points, it corresponds to $(-2) \in \pi_0(\mathds{S}) \cong \Z$ \cite[Proposition 3.1.3]{etaRealRealization}. Here $\mathds{S}$ is the topological sphere spectrum. 
\end{remark}
\begin{remark}\label{etaMGL}
    We have the cofiber sequence 
    \[ \A^2 \backslash \{0\} \xrightarrow{\tilde{\eta}} \mathds{P}^1 \to \mathds{P}^2\]
    Therefore $\cofib(\eta) =: \mathds{1}/\eta \cong \Sigma^{-2,-1} \Sigma^{\infty} \mathds{P}^2$. Here $\A^2 \backslash \{0\}$ can be considered to be pointed at $(1,0)$, and $\mathds{P}^1, \mathds{P}^2$ are taken to be pointed at $[1:0] := \infty$. The unit map of $\MGL$, at the first level is just the inclusion $\mathds{P}^1 \hookrightarrow \MGL_1 \cong \mathds{P}^{\infty}$. Therefore, we get a canonical factorization 
    \[\Sigma^{-2,-1} \Sigma^{\infty} \mathds{P}^1 \to \Sigma^{-2,-1} \Sigma^{\infty} \cofib(\tilde{\eta}) \to \MGL\] 
    In other words, the unit map $\mathds{1} \to \MGL$ has the following factorization
\[\begin{tikzcd}
	{\mathds{1}} & {} & \MGL \\
	& {\mathds{1}/\eta}
	\arrow[from=1-1, to=1-3]
	\arrow[from=1-1, to=2-2]
	\arrow[from=2-2, to=1-3]
\end{tikzcd}\]
\end{remark}
This map $\mathds{1}/\eta \to \MGL$ is determined by the maps \[\Sigma^{2r-2, r-1} \cofib(\tilde{\eta}) \to \MGL_r\] Following \cite[Lemma 3.7]{Hoyoismalgmot}, one can show that 
\[(\Sigma^{\infty}\Sigma^{2r-2, r-1} \cofib(\tilde{\eta}))_{\leq r} \xrightarrow{\cong} (\Sigma^{\infty}{\MGL_r})_{\leq r}\]
So, we get that
\[(\mathds{1}/\eta)_{\leq 0} \xrightarrow{\cong} (\MGL)_{\leq 0}\]
In particular, as $\mathds{1}/\eta$ is connected, we have 
\[\underline{\pi}_0(\mathds{1}/\eta) \cong \underline{\pi}_0(\MGL)\] 
\subsection{Slice filtration}
In classical stable homotopy theory Postnikov $t$-structure, and its associated filtration give rise to Atiyah-Hirzebruch spectral sequence \cite{Maunder_1963}. The question of constructing a similar picture in the motivic realm led to the introduction of the slice filtration. The filtration induced by the homotopy $t$-structure was not a good candidate, as it does not take into account suspension with respect to $\Gm$.
 Voevodsky constructed slice filtration in the stable motivic homotopy category $\SH(S)$ \cite{open_problems}.\\[1.0ex]
 Let $\SH^{\text{eff}}(S)$ be the smallest colimit closed triangulated subcategory of $\SH(S)$ containing $\{\Sigma^{\infty} X_+ : X \in \SmS_S\}$. We get the following filtration
 \[... \hookrightarrow \Sigma^{q+1}_{\mathds{P}^1} \SH^{\text{eff}}(S) \hookrightarrow  \Sigma^{q}_{\mathds{P}^1} \SH^{\text{eff}}(S) \hookrightarrow... \hookrightarrow \SH(S)\]
 The inclusions $i_q : \Sigma^{q}_{\mathds{P}^1} \SH^{\text{eff}}(S) \hookrightarrow \SH(S) $ preserve colimits. Therefore using \cite[Corollary 5.5.2.9]{HTT}, they have right adjoints $r_q$. Let's define $f_q:= i_q \circ r_q$.\\[1.0ex]
 The unit of the adjunction $\textit{Id} \xrightarrow{\cong} r_q \circ i_q$ is an isomorphism. The counit $i_{q+1} \circ r_{q+1} \to \textit{Id}$, applied on $f_q$ gives a map 
 \begin{align*}
     f_{q+1} \circ f_q \to f_q\\
     f_{q+1} \to f_{q}
 \end{align*}
 For any $E \in \SH(S)$, there is a decreasing filtration
 \[... \to f_{q+1}E \to f_{q}E \to f_{q-1}E \to ....\]
We call $f_q E$, the $q$-th effective cover of $E$. The associated graded of this filtration\\ $\cofib (f_{q+1}E \to  f_{q}E)=:s_{q}E$ is called the \textbf{$q$-th slice} of $E$. We will use the following basic properties of slice functor $s_q$ \cite[Section 2]{slicesofhermitianktheory}.\begin{itemize} \label{slices} 
    \item $s_q$ is triangulated
    \item $s_q(\Sigma^{2,1} E) \cong \Sigma^{2,1} s_{q-1} (E)$
    \item for $E \in \Sigma^{q+1}_T \SH(S)^{\eff}$, $s_q(E) \cong \ast$ in $\SH(k)$
    \item Following \cite{coloredOperads}, there are natural maps 
    \[s_q E \wedge s_t F \to s_{q+t} (E \wedge F)\]
    In other words, the \say{full slice functor} $s_* : \SH(S) \to \SH(S)^{\Z}$ is lax symmetric monoidal. In particular, $s_0$ preserves ring structure. 
\end{itemize} 
The slices of any spectrum $E$ are modules over the zero slice of the sphere spectrum $\mathds{1}$. Over any field $s_0(\mathds{1}) \cong \HZ$ \cite{coniveau}. Bachmann-Hoyois showed that $s_0(\mathds{1}) \cong \HZ$ in $\SH(S)$, where $S$ is essentially smooth over a Dedekind domain \cite[Theorem B.4]{normsinmotivichomotopy}. In this case $\HZ$ denotes Spitzweck's motivic cohomology \cite{SpitzweckMotCoh}.
\begin{remark}\label{SlicesMGL}
  Changing the geometric argument of \cref{etaMGL} a bit, or following \cite[Theorem 3.1]{slices_MGL}, one can show that the cofiber of the unit map $\mathds{1} \to \MGL$ lies in $\Sigma^{2,1}\SH(S)^{\text{eff}}$. So, we get 
  \[s_0(\MGL) \cong s_0(\mathds{1}) \cong \HZ\]
  Spitzweck computed all slices of $\MGL$ \cite{slices_MGL}, if Hoyois-Hopkins-Morel isomorphism \cite[Theorem 7.12]{Hoyoismalgmot} holds. The following holds over fields of exponential characteristic $e$, after inverting $e$.
  \[s_q(\MGL) \cong \Sigma^{2q,q} \HZ \otimes \text{MU}_{2q}\]
  Following \cite[Lemma 3.24]{annals_stable} the cofiber of the map $\mathds{1}/\eta \to \MGL$ lies in $\Sigma^{4,2} \SH(S)^{\text{eff}}$. In particular, $s_1(\mathds{1}/\eta) \cong s_1(\MGL)$.
\end{remark}
\subsection{Orientation} \label{orientationSection}
In $\A^1$-homotopy theory, different notions of orientations of cohomology theories are quite well studied. Here we will recall some basic facts about oriented and special linear oriented theories primarily following \cite{PaninOriented2}, \cite{PW}, and \cite{ananyevskiy2019sloriented}. The content of this section is standard. We have included this for the sake of being notationally coherent. 
\subsubsection{Oriented cohomologies}
Panin and Smirnov \cite{PaninOriented2}, while discussing oriented cohomology theories of varieties did not only consider the theories representable in the stable motivic homotopy category $\SH(k)$. They defined cohomology theories for smooth varieties using Eilenberg-Steenrod like axioms.\\[1.01ex] 
Let $\textit{SmOp}$ denote the category of smooth pairs $(X,U)$, where $X$ is a smooth scheme over a fixed field $k$, and $U \hookrightarrow X$ is an open immersion. In this sense a cohomology theory is a contravariant functor $E: \textit{SmOp} \to \textit{Ab}$, and a functorial boundary map $\partial : E(U) \to E(X,U)$ having suitable localization, excision, and homotopy invariance properties. \cite[Definition 1.1]{PaninOriented2}. Any bigraded theory representable in $\SH(k)$ is a cohomology theory in the sense of Panin-Smirnov. As for this paper, we are only interested in such theories, so we will refrain from discussing non-representable theories.\\[1.01ex] 
Let $E$ be a commutative monoid in $(\SH(S), \wedge, \mathds{1})$. For any smooth pair $(X,U)$, with open immersion $U \xhookrightarrow{i} X$, $E^{*,*}(X,U)$ is defined to be $E^{*,*}(Th(\mathcal{N}_i))$.
\begin{definition}\label{orientation}
    Let $X \in \SmS$. A theory $E$ is said to have an \textbf{orientation} if for every rank $n$ vector bundle $p: V \to X$, there exists an element \\$\thom \in E^{2n,n}(Th(V))$, that satisfies the following properties. 
    \begin{itemize}
        \item For a vector bundle isomorphism $\phi: V \to V'$, we get 
        \[\thom (V) = \phi^E(\thom(V'))\]
        \item For a morphism $f: Y \to X$ of smooth varieties,
        \[f^E(\thom(V)) = \thom(f^*(V))\]
        \item We have \textbf{Thom isomorphism}. In other words, the map $E^{*,*}(X) \to E^{*+2n,*+n}(Th(V))$, taking $a \to p^*(a) \cup \thom(V)$ is an $E^{*,*}(X)$ module isomorphism.
        \item Let $q_i : V_1 \oplus V_2 \to V_i$ denote the projection map for $i \in \{1,2\}$. We have the following multiplicative property of Thom classes 
        \[{q_1}^*\bigl(\thom(V_1)\bigr) \cup {q_2}^*\bigl(\thom(V_2)\bigr) = \thom (V_1 \oplus V_2) \in E^{*,*}\bigl(Th(V_1 \oplus V_2)\bigr)\]
        Here we have actually considered the maps induced by $q_i$'s on Thom spaces.
        \item The unit of $E$ gives rise to an element $1 \in E^{0,0}(\text{spec}(S)_{+})$. Applying $T$-suspension on this, we get an element $\Sigma_T (1) \in E^{2,1} (T)$. For the trivial line bundle $\mathcal{O}_X$, we have $\thom(\mathcal{O}_X) = \Sigma_T (1)$.
    \end{itemize}
\end{definition}
\begin{definition}\label{chernClass}
    A theory $E$ is endowed with a theory of \textbf{Chern classes} for line bundles, if for every $X \in \SmS$, and any line bundle $L \to X$, there is an element $c(L) \in E^{2,1}(X)$ with the following properties
    \begin{itemize}
        \item \textbf{functoriality} For isomorphic line bundles $L_1$, and $L_2$; $c(L_1) = c(L_2)$. For any morphism $f: Y \to X$, $f^E(c(L)) = c(f^*(L))$.
        \item \textbf{nondegeneracy} The following map is an isomorphism      \[(1, c(\mathcal{O}(-1))): E^{*,*}(X) \oplus E^{*,*}(X) \to E^{*+2,*+1}(X \times \mathds{P}^1)\]
        As usual, $\mathcal{O}(-1)$ denotes the tautological line bundle on $\mathds{P}^1$. 
        \item \textbf{vanishing} The Chern class of the trivial line bundle vanishes.
        \[c(\mathcal{O}_X) = 0\]
    \end{itemize}
\end{definition}
\begin{definition} \label{thomLineBundle}
    A theory $E$ is endowed with a theory of \textbf{Thom classes} for line bundles, if for every $X \in \SmS$, and any line bundle $L \to X$, there is an element $\thom (L) \in E^{2,1}(Th(L))$, satisfying the following properties
    \begin{itemize}
        \item For any isomorphism of line bundles $\phi: L_1 \to L_2$, $\phi^E(\thom (L_2)) = \thom (L_1)$.
        \item For a scheme morphism $f: Y \to X$, $f^E(\thom (L)) = \thom (f^*(L))$.
        \item Let's consider the line bundle $p:\A^1 \times X \to X$. The following map is an isomorphism.
        \[p^*(-) \cup \thom (\mathcal{O}_X) : E^{*,*}(X) \to E^{*+2, *+1}(Th(X \times \mathds{A}^1)\]
        
    \end{itemize}
\end{definition}
\begin{remark} \label{ChernThomLineEqui}
    As described in \cite{PaninOriented2}, there is a one-to-one correspondence between Chern and Thom classes of line bundles. 
    Let a theory $E$ be endowed with Thom classes $\thom(L)$ for line bundles $L \to X$. We denote the zero section by $z: X \to L$, and the complement of the zero section of $L$ by $L^0$.\\[1.000ex]
    Consider the map (in $\textit{SmOp}$), $i: (L,\emptyset) \to (L,L^0)$. We know $E^{*,*}(Th(L)) := E^{*,*}(L,L^0) := E^{*,*}_X(L)$, where $E^{*,*}_X(L)$ is $E$-cohomology on $L$ with support on $X$. We denote the support extension map $E^{*,*}_X(L) \to E^{*,*}(L)$ by $i^E$.\\[1.000ex]
    We can define the corresponding Chern class as
    \[c(L) = z^E\biggl(\Bigl(i^E\bigl(\thom(L)\bigr)\Bigr)\biggr)\]
    Now, let's assume that the theory $E$ is endowed with Chern classes $c(L)$ for line bundles $L \to X$. Consider the vector bundle $V = L \oplus \mathds{1}$. Let $\mathds{P}(V)$ denote the space of lines in the vector bundle $V$. We denote $\mathds{P}(V) - \mathds{P}(\mathds{1})$ by $U$. We have the natural projection $p: \mathds{P}(V) \to X$, and inclusion $i: (L, L^0) \hookrightarrow (\mathds{P}(V), U)$. We set the corresponding Thom class of $L$ to be
    \[\thom(L) = i^E(c(\mathcal{O}_{\mathds{P}(V)}(1) \otimes p^*(L)))\]
    These two constructions are inverses to each other \cite[1.2.2]{PaninOriented2}.
    \end{remark}
  \begin{theorem}[Projective bundle theorem] \label{pbt} \cite[Theorem 3.9]{PaninOriented2}
    Let $E$ be a theory endowed with Chern classes for line bundles. Let $V \to X$ be a rank $n$ vector bundle on some $X \in \SmS_S$. Let $\xi = c(\mathcal{O}_{\mathds{P}(V)}(-1)) \in E^{*,*}(\mathds{P}(V))$. Then $E^{*,*}(\mathds{P}(V))$ is a finite rank $E^{*,*}(X)$ module with basis $\{1, \xi, \xi^2,..., \xi^{n-1}\}$. The following map is an isomorphism
    \[(1,\xi,\xi^2,..., \xi^{n-1}): E^{*,*}(X) \oplus E^{*,*}(X) \oplus... \oplus E^{*,*}(X) \to E^{*,*}(\mathds{P}(V))\]
    For a trivial rank $n$ vector bundle $V$, $\xi^n = 0$. This works if we take $\xi$ to be $c(\mathcal{O}_{\mathds{V}}(1)) \in E^{*,*}(\mathds{P}(V))$ as well.
    
\end{theorem}  
\begin{remark}\label{orientationEquivalence}
 Using \cref{pbt}, any theory with Chern classes for line bundles can be endowed with (higher) Chern classes for all vector bundles \cite[Theorem 3.27]{PaninOriented2}. Using the notations from \cref{pbt}, there are unique elements \\$c_i(V) \in E^{2i,i}(X)$ such that 
    \[\xi^n - c_1(V) \xi^{n-1} +...+ (-1)^{n} c_n (V) = 0\]
    We set $c_0(V) = 1$, and $c_m(V) = 0$ for $m > n$. In this way, we obtain Chern classes $c_i(V)$, for any rank $n$ vector bundle $V \to X$, over any $X \in \SmS$, with the following properties
    \begin{itemize}
        \item We have $c_0(L) = 1$, and $c_1(L) = c(L)$, for any line bundle $L$.
        \item For isomorphic vector bundles $V$ and $V'$, $c_i(V) = c_i(V')$.\\
        For each morphism $f: Y \to X$ of smooth varieties,
        \[f^E(c_i(V)) = c_i(f^*(V))\]
        \item For any short exact sequence of vector bundles
        $0 \to V_1 \to V \to V_2 \to 0$, we have the following formula known as Cartan formula
        \[c_i(V) = c_i(V_2) + c_1 (V_1) \cup c_{i-1} (V_2) +...+ c_{i-1} (V_1) \cup c_{1} (V_2) + c_i(V_1) \]
        \item $c_i(V)$ is nilpotent for $i \geq 1$ by construction.
    \end{itemize}
We have already seen a one-to-one correspondence between Chern and Thom classes of line bundles. If a theory $E$ has orientation, that is Thom classes for each vector bundle, of course, it is endowed with Thom classes for line bundles. Given Chern classes for line bundles, we can construct Thom classes for every rank $n$ vector bundle $V \to X$ \cite[Theorem 3.35]{PaninOriented2}.
\end{remark}
\begin{remark}
    Let $E$ be a commutative monoid in $\SH(S)$. We can orient $E$ with very little data \cite[Definition 1.2.1, Section 2]{PPR}.\\[1.000ex] 
    Consider the space $\mathds{P}^{\infty} := \ncolim_n \mathds{P}^{n}$, pointed at $g: \mathds{P}^{0} \hookrightarrow \mathds{P}^{\infty}$. We denote $\Taut(1) := \mathcal{O}_{\mathds{P}^{\infty}}(-1) := \ncolim_n \mathcal{O}_{\mathds{P}^{n}}(-1)$. Its fiber over the point $g \in \mathds{P}^{\infty}$ is $\A^1$. We denote its zero section by $z : \mathds{P}^{\infty} \hookrightarrow \Taut(1)$\\[1.000ex]
     A Thom orientation of  $E$ can be thought of as a choice of an element\\ $th \in E^{2,1}(Th(\Taut(1))$, such that its restriction over the distinguished point $g$ is $\Sigma_T 1 \in E^{2,1}(T)$. Similarly, a Chern orientation of  $E$ can be thought of as a choice of an element $c \in E^{2,1}(\mathds{P}^{\infty})$, such that its restriction over $\mathds{P}^{1}$ is $\Sigma_{\mathds{P}^{1}} 1$. A Chern and Thom orientation correspond to each other if $c = z^*(th)$.\\[1.01ex]
     Let $S$ be the spectrum of a field. In the unstable $\A^1$-homotopy category $\mathcal{H}(S)_{*}$, the Picard group fuctor for smooth schemes $\textit{Pic}$ is represented by $\mathds{P}^{\infty}$ \cite[Corollary 4.2]{motcoh_voevodsky_weibel_mazza}. There is a functor isomorphism\\$\hom_{\Spc(S)_{\bullet}}(-, \mathds{P}^{\infty}) \to \textit{Pic}(-)$, sending $id: \mathds{P}^{\infty} \to \mathds{P}^{\infty}$ to the class of $\Taut(1)$. For a line bundle $L \to X$, let's consider the map $f: X \to \mathds{P}^{\infty}$ in $\mathcal{H}(S)_{*}$, corresponding to the class of $L$ in $\textit{Pic}(X)$. We set $c(L) = f^*(c) \in E^{2,1}(X)$. In this way, we get an orientation on $E$.    
\end{remark}
\begin{remark} \label{orientingMGL}
   The zero section map $z: \mathds{P}^{\infty} = \text{BGL}_1 \to \Taut(1) = \text{MGL}_1$ induces a weak equivalence \cite[Lemma 3.2]{vezzosi_BP}. The normal bundle of the closed immersion $\mathds{P}^{n} \hookrightarrow \mathds{P}^{n+1}$ is the canonical line bundle $\mathcal{O}_{\mathds{P}^{n}}(-1)$, and $\mathds{P}^{n+1} - \mathds{P}^{n} \cong \A^{n+1}$. In particular, we have $Th(\mathcal{O}_{\mathds{P}^{n}}(-1)) \cong \mathds{P}^{n+1}$. Passing to the colimit we are done.\\[1.000ex]
   So, we have a canonical map $\varphi: \Sigma^{-2,-1} \mathds{P}^{\infty} \to \MGL$. The corresponding element $\varphi \in \MGL^{2,1}(\mathds{P}^{\infty}) \cong \MGL^{2,1}\Bigl(Th\bigl(\Taut(1)\bigr)\Bigr)$ can be taken as a choice of a Thom element.
\end{remark}
\begin{lemma}\label{gysin} \cite{Gysinnenashev} \cite{PaninOriented2}
    Let $i: Y \hookrightarrow X$ be a closed immersion of codimension $d$, and $E$ be an oriented theory. There is a push forward map \[i_*: E^{*,*} (Y) \xrightarrow{\thom_{\mathcal{N}_i}} E^{*+2d,*+d} (Th(\mathcal{N}_i)) := E^{*+2d,*+d}_{Y} (X) \xrightarrow{\tilde{i}^E} E^{*+2d,*+d} (X)\]
    Here $\tilde{i}^E$ is the support extension map. The composition $i_*$ is known as the Gysin map.
\end{lemma}
\begin{lemma} \label{OrientedCohomologyOfGrassmannian} \cite[Theorem 2.0.7]{PPR}, \cite[Proposition 3.4]{vezzosi_BP}
 Let $E$ be an oriented theory. We denote the infinite Grassmannian of $n$-planes, as described in \cref{constructMGL}, by $\BGL_n$. Then, we have a description of the $E^{*,*}$-cohomology ring of $\BGL_n$ as the following power series ring.
 \[E^{*,*}(\BGL_n) \cong E^{*,*}(S) \llbracket c_1, c_2,...,c_n \rrbracket\]
 Here $c_i$ denotes the $i$-th Chern class of the tautological $n$-bundle $\mathcal{T}_n$ with respect to the theory $E$.\\[1.01ex]
 The map $i_n^*$, induced by the inclusion $i_n: \BGL_n \hookrightarrow \BGL_{n+1}$ satisfies \\$i_n^*(c_m) = c_m$ for $m \leq n$, and $i_n^*(c_{n+1}) = 0$. 
\end{lemma}
Just like $\MU$ in topology, $\MGL$ is the universal oriented cohomology theory in $\SH(S)$.
\begin{theorem} \label{uniMGL} \cite[theorem 2.3.1]{PPR}
    Let $S$ be a Noetherian base scheme of finite Krull dimension. For a (homotopy) commutative ring spectrum $E$ in $\SH(S)$, the set of orientations on $E$ is in bijective correspondence with ring spectrum map $\MGL \to E$ in $\SH(S)$. 
\end{theorem}
In \cite{PPR}, \cref{uniMGL} is stated over fields. The proof works over general base schemes.
\begin{remark} \label{FormalGroupLaw}
Let $E$ be an oriented theory with an orientation $\omega$. We have $E^{*,*}(\mathds{P}^{\infty}) \cong E^{*,*}(S) \llbracket {c_1^{\omega}}\rrbracket$.\\
Let's consider the two projection maps $p_i: \mathds{P}^{\infty} \times \mathds{P}^{\infty} \to \mathds{P}^{\infty}$ for $i \in \{1,2\}$. We can identify $E^{*,*}(\mathds{P}^{\infty} \times \mathds{P}^{\infty})$ as \[E^{*,*}(S) \llbracket p_1^*(c_1^{\omega}), p_2^*(c_1^{\omega}) \rrbracket = E^{*,*}(S) \llbracket u_1, u_2 \rrbracket\]
We set \[F^{\omega} (u_1, u_2) = c^{\omega}(p_1^*(\mathcal{O}(-1)) \otimes p_2^*(\mathcal{O}(-1)) \in E^{*,*}(S) \llbracket u_1, u_2 \rrbracket\]
$F^{\omega}$ is a formal group law over $E^{*,*}(S)$. As the (bi)degree of the element $c_1^{\omega}$ is $(2,1)$, it is in fact a formal group law over the even subring $\oplus_{n} E^{2n,n}(S)$ of $E^{*,*}(S)$.\\[1.00000ex]
Let $L_1 \to X$, and $L_2 \to X$ be line bundles over some $X \in \SmS_S$. From \cite[Proposition 1.15]{PaninOriented2}, we have
\[c_1^{\omega}(L_1 \otimes L_2) = F^{\omega}(c_1^{\omega}(L_1), c_1^{\omega}(L_2))\]

\end{remark}
\subsubsection{Special linear and symplectic oriented cohomologies}
One well-studied example of a non-orientable theory is Balmer's derived Witt theory ($\text{W}$). In the last decade, quadratic refinements of the classical theories like Chow-Witt theory, Milnor-Witt motivic cohomology, and Hermitian $K$-theory were studied extensively. Projective bundle formula fails for them in general. They all have a weaker form of orientation, namely special linear orientation \cite{ananyevskiy2019sloriented}.
\begin{definition}
    A vector bundle $V$ over $X$ is called special linear if its determinant bundle is trivialized. That means, it comes with a choice of vector bundle isomorphism $\lambda : det(V) \xrightarrow{\cong} \mathcal{O}_X$ \cite[section 5]{PW}.\\[1.000ex]
    An isomorphism $\phi$ between special linear bundles $(V, \lambda)$, and $(V', \lambda')$ is a vector bundle isomorphism $\phi: V \to V'$, such that $\lambda'(det (\phi)) = \lambda$. In other words, the lower square in the following diagram commutes.
\[\begin{tikzcd}
	V && {V'} \\
	{det(V)} && {det(V')} \\
	{\mathcal{O}_X} && {\mathcal{O}_X}
	\arrow[from=1-1, to=2-1]
	\arrow["\phi", from=1-1, to=1-3]
	\arrow[from=1-3, to=2-3]
	\arrow["{det(\phi)}", from=2-1, to=2-3]
	\arrow["\lambda"', from=2-1, to=3-1]
	\arrow["{\lambda'}", from=2-3, to=3-3]
	\arrow["id"', from=3-1, to=3-3]
\end{tikzcd}\]
\end{definition}
\begin{definition}\label{SLorientation}
   Let's consider a ring cohomology theory represented by a $T$-spectrum $E$. It is said to have a normalized \textbf{special linear ($\SL$) orientation} if for every rank $n$ special linear vector bundle $p: V \to X$, there exists an element $\thom \in E^{2n,n}(Th(V))$, that satisfies the properties described in \cref{orientation}. 
\end{definition}
This class $\thom \in E^{2n,n}(Th(V))$, giving the Thom isomorphism, is called the Thom class of the $\text{SL}$-bundle $V$. Following the notation from \cref{ChernThomLineEqui}, the \textit{Euler class} of $V$ is defined as
\[e(L) = z^E\biggl(\Bigl(i^E\bigl(\thom(L)\bigr)\Bigr)\biggr) \in E^{2n,n}(X)\]
Panin-Walter \cite{PW} constructed, and proved the following universal property of the spectrum $\MSL$.
\begin{theorem}\cite[Theorem 5.5, Theorem 5.9]{PW}\label{MSLuniversal}
    Let the base scheme $S$ admit an ample family of line bundles. A commutative monoid morphism $\MSL \to E$ in $\SH(S)$ determines a natural $\SL$-orientation on $E$.\\[1.01ex] 
    Every $\SL$-orientation on $E$ determines a morphism $\MSL \to E$ that preserves the Thom classes for every special linear bundle. This morphism is unique modulo a particular subgroup of $\Hom_{\SH(S)}(\MSL,E)$.\\[1.0000ex]
    More precisely, let $E$ be given a (normalized) $\SL$-orientation with Thom classes $\thom^E(V,\lambda)$, for every special linear bundle $(V,\lambda)$ over some smooth scheme $X$. There exists a morphism $\phi: \MSL \to E$, such that 
    \[\phi (\thom^{\MSL}(V,\lambda)) = \thom^E(V,\lambda)\]
    There are some obstructions for $\phi$ to be a morphism of monoids. 
\end{theorem}
\begin{remark}\label{MSptoMSL}
    A rank $(2n)$ \textit{symplectic} ($\text{Sp}$) bundle over $X$ is a pair $(V, \phi)$, where $V \to X$ is a rank $2n$ vector bundle, and $\phi$ is a given symplectic form. For any symplectic bundle $(V, \phi)$, the Pfaffian induces an isomorphism $det(V) \to \mathcal{O}_X$.\\[1.0000ex]
    Let $G$ denote one of the classical groups $\GL$, $\SL$, or $\text{Sp}$. A cohomology theory $E$ is said to have $G$-orientation if $E$ has Thom isomorphism for all $G$-bundles. Panin-Walter \cite{PW} constructed algebraic symplectic bordism $\text{MSp}$. This is the universal $\text{Sp}$-oriented cohomology theory. Normalized $\text{Sp}$-orientation can be defined specifying Thom classes for symplectic bundles, which satisfy the properties described in \cref{orientation}. 
    \begin{theorem}\label{MSpUniversal}\cite[Theorems 12.2, 13.2]{PW}
     Let $E$ be a commutative monoid in $(\SH(S), \wedge, \mathds{1})$, endowed with a (normalized) $\text{Sp}$-orientation, given by Thom classes $\thom^E(V,\phi)$, for every symplectic bundle $(V,\phi)$ over some $X \in \SmS_S$. There exists a unique morphism of commutative monoids $\phi: \text{MSp} \to E$, such that 
    \[\phi (\thom^{\text{MSp}}(V,\phi)) = \thom^E(V,\phi)\]  
    Every monoid morphism $\text{MSp} \to E$ determines a unique $\text{Sp}$-orientation of $E$.
    \end{theorem}
\end{remark}
\begin{remark}\label{newSpvanish}
 Let $\nu \in \pi_{3,2}(\mathds{1})$ be the second stable hopf map. It can be constructed via the following construction on $(\SL_2,1)$.
 \[S^{7,4} \cong \Sigma^{1,0} \SL_2 \wedge \SL_2 \xleftarrow{\cong} \SL_2 \ast \SL_2 \to \Sigma^{1,0} \SL_2 \times \SL_2 \to \Sigma^{1,0} \SL_2 \cong S^{3,2}  \]
 Ananyevskiy \cite[Lemma 5.3]{Ananyevskiy2020ThomII} showed that for any $\text{Sp}$-oriented commutative ring spectrum $E$, the map $u_*: \pi_{3,2}(\mathds{1}) \to \pi_{3,2}(E)$, induced by the unit map $\mathds{1} \to E$, takes $\nu$ to $0$.
\end{remark}
\section{An Interpolation between special linear and general algebraic cobordism}\label{An Interpolation between special linear and general algebraic cobordism Section}
\subsection{Main Question} \label[subsection]{question}
We have introduced the canonical map between $\MSL$ and $\MGL$ in \cref{Remark3.1}
\[ \Phi:\MSL \to \MGL\]
From \cref{orientingMGL}, $\MGL_1 \cong \mathds{P}^{\infty}$, we have an orientation map for $\MGL$ \[\varphi:\Sigma^{-2,-1}\Sigma^{\infty} \mathds{P}^{\infty} \to \MGL\]
As $\MGL$ is a ring spectrum, combining these we get a map \begin{equation} \label{main_question}
   \MSL \wedge \Sigma^{-2,-1}\Sigma^{\infty} \mathds{P}^{\infty} \xrightarrow{\Phi \wedge \varphi} \MGL \wedge \MGL \xrightarrow{\mu} \MGL 
\end{equation}
Let's name the composite map in \cref{main_question} $\Psi$. \textbf{Our goal is to prove that $\Psi$ is an isomorphism in $\SH(S)$}.\\[1.000ex]
If $S$ is $Spec(\C)$, as complex realization preserves derived smash product, we get back Conner-Floyd's map \cref{msuCFIntro}.
\subsection{Computing oriented cohomology rings of infinite special linear Grassmannians } \label[subsection]{BSLn}
Let $E$ be an oriented ring cohomology theory represented in $\SH(S)$. We compute $E^{*,*}(\BSL_n)$. The idea is to use a variant of the cofiber sequence \cref{cofiber_SGr}.\\[1.00000ex]
As $det(\mathcal{T}_{n,np})$ is the pullback of $\mathcal{O}(-1)$ on $\mathds{P}^{{np \choose p} -1}$ along the Pl\"ucker embedding of $\GL_{n,np}$ \cite[B.5.7]{Fulton}, $det(\mathcal{T}_{n,np})$ is denoted as $\mathcal{O}(-1)_{\Gr_{n,np}}$.\\[1.0000ex]
We consider the zero section $z: \Gr_{n,np} \to \mathcal{O}(-1)_{\Gr_{n,np}}$. Following \cref{gysin}, we can use the Gysin sequence for the following closed immersion \[z: \Gr_{n,np} \to \mathcal{O}(-1)_{\Gr_{n,np}}\] 
In this case the Gysin map $z_{*}: E^{*,*}(\Gr_{n,np}) \to E^{*+2,*+1} (\mathcal{O}(-1)_{\Gr_{n,np}})$ is the same as $\cup th$ (Thom isomorphism map) composed with the support extension map, as explained below.\\[1.000ex]
Let's consider the map $i:(\mathcal{O}(-1)_{\Gr_{n,np}}, \phi) \to (\mathcal{O}(-1)_{\Gr_{n,np}}, (\mathcal{O}(-1)_{\Gr_{n,np}})^0)$. The induced map $i^E$ is called the support extension map. So, we have the following long exact sequence-\\[1.5ex]
\[\begin{tikzcd}[column sep=scriptsize]
	&& {E^{*+2,*+1} (Th(\mathcal{O}(-1)_{\Gr_{n,np}}))} \\
	{...} & {E^{*,*}(\Gr_{n,np})} & {E^{*+2,*+1}_{\Gr_{n,np}} (\mathcal{O}(-1)_{\Gr_{n,np}})} & {E^{*+2,*+1} (\mathcal{O}(-1)_{\Gr_{n,np}})} \\
	&&& {{E^{*+2,*+1}(\SGr_{n,np})}} & {...}
	\arrow[from=2-1, to=2-2]
	\arrow["{\cup th}", from=2-2, to=2-3]
	\arrow["{:=}"', from=2-3, to=1-3]
	\arrow["{i^E}", from=2-3, to=2-4]
	\arrow[from=2-4, to=3-4]
	\arrow[from=3-4, to=3-5]
\end{tikzcd}\]
It can be written as-\\[1.5ex]
 \begin{equation} \label{finitecohomologycomp}
\begin{tikzcd}[column sep=tiny,row sep=scriptsize]
	{....} & {E^{*+2,*+1} (Th(\mathcal{O}(-1)_{\Gr_{n,np}}))} \\
	{E^{*,*}(\Gr_{n,np})} & {E^{*+2,*+1}_{\Gr_{n,np}} (\mathcal{O}(-1)_{\Gr_{n,np}})} & {E^{*+2,*+1} (\mathcal{O}(-1)_{\Gr_{n,np}})} \\
	&&& {E^{*+2,*+1}(\SGr_{n,np})} \\
	& {} & {E^{*+2,*+1}(\Gr_{n,np})} & {...}
	\arrow["{\cup th}", from=2-1, to=2-2]
	\arrow["{:=}"', from=2-2, to=1-2]
	\arrow["{i^E}", from=2-2, to=2-3]
	\arrow[from=2-3, to=3-4]
	\arrow[from=3-4, to=4-4]
	\arrow["{\cong_{z^*}}", from=2-3, to=4-3]
	\arrow["\Phi"', from=2-1, to=4-3]
	\arrow[from=4-3, to=3-4]
	\arrow[from=1-1, to=2-1]
\end{tikzcd}
\end{equation}
As described in \cref{ChernThomLineEqui}, for a line bundle $Y \to X$, and a choice of Thom class $\Th \in E^{2,1}_X(Y)$, $z^* \circ i^E (th)$ is the corresponding Chern class $c_1(Y)$. From now onwards, we consider cohomology rings of all concerned spaces. From \cref{OrientedCohomologyOfGrassmannian}, we have \[E^{*,*}(\colim_p (\Gr_{n,np})) \cong E^{*,*}(\BGL_n) \cong E^{*,*}(S) \llbracket c_1,..,c_n \rrbracket\]
Here $c_i$ is the $i$-th Chern class of $\mathcal{T}_{n,np}$ with respect to the theory $E$. We have denoted $E^{*,*}(S)$ (or $E_{*,*}(S)$) as $E^{*,*}$ (or $E_{*,*}$). In the inverse system $...\to E^{*,*}(\Gr_{n,n(p+1)}) \to E^{*,*}(\Gr_{n,np})\to E^{*,*}(\Gr_{n,n(p-1)}) \to....$, the pullback maps are surjective \cite[Remark 2.0.6]{PPR}.
Let's consider a short exact sequence of inverse system of Abelian groups
$(A_i) \to (B_i) \to (C_i)$.
If $(A_i)$ satifies Mittag-Leffler conditions, $0 \to \nlim_i(A_i) \to \nlim_i(B_i) \to \nlim_i(C_i) \to 0$ is also exact \cite[Lemma 12.31.3]{stacks-project_Henselization}. Direct sums are exact in Abelian groups.

Taking limit along $p$ in the diagram \cref{finitecohomologycomp}, and considering cohomology rings, we have got the following\\ [2ex]
\begin{equation}\label{lesForLemma}
\begin{tikzcd}[column sep=small]
	{..} & {E^{*,*}\llbracket c_1,..,c_n \rrbracket} & {E^{*,*}\llbracket c_1,..,c_n \rrbracket} & {E^{*,*}\llbracket c_1,..,c_n \rrbracket} & {E^{*,*}(\SGr_n)} & {...} \\
	&&& {E^{*,*}\llbracket c_1,..,c_n \rrbracket}
	\arrow[from=1-1, to=1-2]
	\arrow["\mathbf{th}", from=1-2, to=1-3]
	\arrow["{i^E}", from=1-3, to=1-4]
	\arrow["\cong_{z^E}", from=1-4, to=2-4]
	\arrow[from=1-4, to=1-5]
	\arrow[from=1-5, to=1-6]
	\arrow["{\otimes c_1 (det(\mathcal{T}_{n,np}))}" ', from=1-2, to=2-4]
\end{tikzcd}
\end{equation}\\[2ex]
Here $\mathbf{th}$ is a module morphism of degree $(2,1)$.
\begin{lemma}
    If $E$ is an oriented cohomology theory with additive formal group law (eg. $\HZ$), the first chern class of a vector bundle of finite rank over some smooth scheme with respect to $E$ is the same as the first chern class of its determinant bundle.
\end{lemma}
\begin{proof}
    Using splitting principle \cite[3.5]{PaninOriented2}, \cite[Lemma 2.2]{splitting_principle}, we can assume the vector bundle $V$ to be direct sum of line bundles. For any $E$ and $V \cong L_1 \oplus L_2 \oplus...\oplus L_n$, 
    \[c_1(V) = c_1(L_1) + c_1(L_2)+...+ c_1(L_n) \]
    In this case, $det(V) \cong L_1 \otimes L_2 \otimes...\otimes L_n$. As $E$ has additive formal group law, $c_1(det(V)) = \sum^n_{i=1} c_1(L_i)$
\end{proof}
Now onwards, we only consider oriented cohomology theories with additive formal group law. For such an oriented cohomology theory $E$, $c_1 (det(\mathcal{T}_{n,np}))$ is the same as $c_1 (\mathcal{T}_{n,np})$. Therefore from \cref{lesForLemma} we have got the following long exact sequence-\\[2ex]
\begin{tikzcd}
	{..} & {E^{*,*}\llbracket c_1,..,c_n \rrbracket} && {E^{*,*}\llbracket c_1,..,c_n \rrbracket} & {E^{*,*}(\BSL_n)} & {...}
	\arrow[from=1-1, to=1-2]
	\arrow[from=1-5, to=1-6]
	\arrow["{\otimes c_1}", from=1-2, to=1-4]
	\arrow["f",from=1-4, to=1-5]
\end{tikzcd}\\ [2ex]
As $c_1$ is a not a zero divisor in ${E^{*,*}\llbracket c_1,..,c_n \rrbracket}$, $\otimes c_1$ is an injective map that sits in a long exact sequence. Therefore the map $f$ has to be surjective. The following lemma summarizes the above computation.
\begin{lemma} \label[lemma]{first_computation}
  If $E$ is an oriented cohomology theory with additive formal group law, $E^{*,*}(\BSL_n) \cong E^{*,*}\llbracket c_2,..,c_n \rrbracket$.   
\end{lemma}
\subsection{Computing oriented (co)homology of $\MSL \wedge \Sigma^{-2,-1} \Sigma^{\infty} \mathds{P}^{\infty}$} \label[subsection]{EofLHS}
Our goal is to show $\Psi \wedge E$ is an isomorphism in $\SH(S)$, for any oriented cohomology theory $E$ with additive formal group law. Therefore, what we want to compute is $E_{*,*}(\MSL \wedge \Sigma^{-2,-1} \Sigma^{\infty}\mathds{P}^{\infty})$. As homotopy colimit is well-behaved with (de)suspension, smash product and homology, we need to compute $\ncolim_n E_{*,*}(\MSL \wedge \Sigma^{-2,-1} \Sigma^{\infty} \mathds{P}^n)$. We know that $\Sigma^{-2,-1} \Sigma^{\infty} \mathds{P}^n$ is strongly dualizable with dual $\Sigma^{-2,-1} \Sigma^{\infty} ({\mathds{P}^n})^{\nu}$ \cite[Lemma 3.18]{HU_picard}. Using the notation from \cite[Section 3]{HU_picard}, we recall that $\nu$ is some virtual normal bundle of $\mathds{P}^n$, and $({\mathds{P}^n})^{\nu}$ is its \say{thomification}.\\
As $E^{*,*}(\mathds{P}^n)$ is a finitely generated free $E^{*,*}$ module, using thom isomorphism for $E$-cohomology, so is $E^{*,*}(({\mathds{P}^n})^{\nu})$.\\[1.5ex] 
Now using \cref{kunnethweuse} we get \[{E_{*,*} (\MSL) \otimes_{E_{*,*}} E_{*,*} (\Sigma^{-2,-1}\Sigma^{\infty} \mathds{P}^{\infty})} \xrightarrow{\cong} {E_{*,*} (\MSL \wedge  \Sigma^{-2,-1}\Sigma^{\infty} \mathds{P}^{\infty})}\]
Using dual Thom isomorphism \[E_{*,*} (\MSL) \cong \ncolim_n E_{*+2n,*+n}(\MSL_n)  \cong \ncolim_n E_{*,*} (\BSL_n)  \cong \ncolim_n E_{*,*} (\ncolim_p \Sigma (\SGr_{n,np}))\]
Similarly, \[E_{*,*} (\Sigma^{\infty} \mathds {P}^{\infty}) \cong \ncolim_{n,p} E_{*,*} (\Sigma^{\infty} \mathds{P}^{np})\]
So, what we want to compute is the following
\begin{equation} \label{whatWeWant}
    \ncolim_n E_{*,*} (\ncolim_p \Sigma^{\infty} (\SGr_{n,np})) \otimes_{E_{*,*}} E_{*,*}(\ncolim_p \Sigma^{\infty} \mathds{P}^{np}) 
\end{equation} 
We know that $\Sigma^{\infty} (\SGr_{n,np})$, and $\Sigma^{\infty} \mathds{P}^{np}$, are both cellular. Using \cite[Proposition 6.1]{motlandweber}, $E^{*,*} (\mathds{P}^{np})$ is flat over $E^{*,*}$. In the proof of \cref{bs1}, we show that $E^{*,*} (\Sigma^{\infty} (\SGr_{n,np}))$ is flat over $E^{*,*}$.\\[1.000ex]
To use \cref{dualmsl}, the condition of vanishing of $\text{lim}^1$ in the following inverse system \[...\to E^{*,*}(X_{i+1}) \to E^{*,*}(X_{i}) \to...\to E^{*,*}(X_{0}))\] is required. This is to make sure that $E^{*,*}(\ncolim_i X_{i}) \cong \nlim_i E^{*,*}(X_{i})$. That is the case for the systems $\{E^{*,*} (\mathds{P}^{np})\}_p$, and $\{E^{*,*} (\SGr_{n,np})\}_p$.\\[1.000ex]
Now, using \cref{dualmsl}, we get that \cref{whatWeWant} is equivalent to the following:
\[\ncolim_n(\Hom_{E_{*,*,c}}(E^{*,*} (\ncolim_p \Sigma (\SGr_{n,np})),E_{*,*}) \bigotimes_{E_{*,*}} \Hom_{E_{*,*,c}}( E^{*,*}(\ncolim_p \Sigma \mathds{P}^{np}), E_{*,*}))\]
\[ \cong \ncolim_n \left(\Hom_{E_{*,*,c}}(\nlim_p E^{*,*} ( \Sigma (\SGr_{n,np})),E_{*,*}) \bigotimes_{E_{*,*}} \Hom_{E_{*,*,c}}(\lim_p E^{*,*}( \Sigma \mathds{P}^{np}), E_{*,*})\right)\]
\[ \cong \ncolim_n\left(\ncolim_p \Hom_{E_{*,*}}( E^{*,*} ( \Sigma (\SGr_{n,np})),E_{*,*}) \bigotimes_{E_{*,*}} \ncolim_p \Hom_{E_{*,*}}( E^{*,*}( \Sigma \mathds{P}^{np}), E_{*,*})\right)\]
\[\cong \ncolim_n \ncolim_p \left( \Hom_{E_{*,*}}( E^{*,*} ( \Sigma (\SGr_{n,np})),E_{*,*}) \bigotimes_{E_{*,*}}  \Hom_{E_{*,*}}( E^{*,*}( \Sigma \mathds{P}^{np}), E_{*,*})\right)\]
Using \cref{bs1} and \cref{bs1.5}, this is the same as the following
\begin{align*}
    & \ncolim_n  \Hom_{E_{*,*,c}} (E^{*,*}\llbracket c_1,c_2,..,c_n \rrbracket, E_{*,*}) & \\
    & \cong \ncolim_n  \Hom_{E_{*,*,c}} (E^{*,*}(\BGL_n), E_{*,*}) & \\
    & \cong   \Hom_{E_{*,*,c}} ( \nlim_n E^{*,*}(\BGL_n), E_{*,*}) & \\
    & \cong   \Hom_{E_{*,*,c}} ( E^{*,*}(\BGL), E_{*,*}) &\\
    & \cong   \Hom_{E_{*,*,c}} ( E^{*,*}(\MGL), E_{*,*}) &\\
    & \cong   \Hom_{E_{*,*,c}} ( E^{*,*}(\Sigma^{2,1}\MGL ), E_{*,*}) & \\
\end{align*}
\begin{lemma} \label[lemma]{bs1}
  We have the following isomorphism \[\lim_p \left(E^{*,*} ( \Sigma (\SGr_{n,np})) \bigotimes_{E_{*,*}}  E^{*,*}( \Sigma \mathds{P}^{np})\right) \cong \lim_p \left(\frac{E^{*,*}[c_2,..,c_n]}{I'} \otimes_{E_{*,*}} \frac{E^{*,*}[c_1]}{c_1^{np+1}} \right)\]
  Here $I'$ is some ideal of $E^{*,*}[c_2,..,c_n]$.
\end{lemma}
\begin{proof}
    From \cite[Remark 2.0.6]{PPR}, or following \cite[Section 6.2]{motlandweber}, $E^{*,*}(\Gr_{n,np})$ is a finite rank free $E^{*,*}$ module, multiplicatively generated by the chern classes $c_1,..,c_n$ of the tautological $n$-bundle $\mathcal{T}_{n,np}$. So, $E^{*,*}(\Gr_{n,np}) \cong \frac{E^{*,*}[c_1,..,c_n]}{I}$, for some ideal $I$. 
From \cref{BSLn}, we have got the following long exact sequence for each ${p,n}$:
\[... \to {E^{r,s}(\Gr_{n,np})} \xrightarrow{\otimes c_1} {E^{r+2,s+1}(\Gr_{n,np})} \to {E^{r+2,s+1}(\SGr_{n,np})} \to {E^{r+2,s+1}(\Gr_{n,np})} \to ...\]
Considering the $E^{*,*}$ cohomology rings, we have the following long exact sequence:
	\[... \to \frac{E^{*,*}[c_1,..,c_n]}{I} \xrightarrow{\otimes c_1} \frac{E^{*,*}[c_1,..,c_n]}{I} \xrightarrow{f} E^{*,*}(\SGr_{n,np}) \xrightarrow{g} \frac{E^{*,*}[c_1,..,c_n]}{I} \xrightarrow{\otimes c_1} \frac{E^{*,*}[c_1,..,c_n]}{I} \to ...\]

From this we get the following short exact sequence of $E^{*,*}$-modules:

	\[0 \to Im(f) \to {E^{*,*}(\SGr_{n,np})} \to {Im(g) } \to 0\]
	
As $E^*$-module $Im(f) \cong \frac{E^{*,*}[c_1,..,c_n]}{I} / Im (\otimes c_1)$ and $Im (g) \cong Ker (\otimes c_1)$. We also have the following short exact sequence:
\[0 \to Im (g) \to \frac{E^{*,*}[c_1,..,c_n]}{I} \to Im (\otimes c_1) \to 0\]

We have $\frac{E^{*,*}[c_1,..,c_n]}{I}$ is a finitely generated free $E^{*,*}$-module, and using \cref{bs2} so is $Im (\otimes c_1)$. So, both $Im(g)$, and $Im(f)$ are projective $E^{*,*}$-module, and $E^{*,*}(\SGr_{n,np})$$ \cong Im(f) \oplus Im(g)$ as $E^{*,*}$-module. In particular $E^{*,*}(\SGr_{n,np})$ is a flat $E^{*,*}$-module.\\[1.0000ex]
Now we have that
\[E^{*,*}(\SGr_{n,np}) \otimes_{E_{*,*}} E^{*,*}(\mathds{P}^{np}) \cong \left(Im(f) \otimes_{E_{*,*}} E^{*,*}(\mathds{P}^{np})\right) \oplus \left(Im(g) \otimes_{E_{*,*}} E^{*,*}(\mathds{P}^{np})\right)\]
Looking at the limiting case, $\lim_p Im(g) \cong \lim_p Ker (\otimes c_1) \cong 0$. As, $E^{*,*}(\mathds{P}^{np})$ is a finitely generated $E^{*,*}$-module, $\lim_p \left(Im(g) \otimes_{E_{*,*}} E^{*,*}(\mathds{P}^{np})\right) = 0$. 
So, for some fixed $n$, we have 
\[\lim_p \left(E^{*,*} ( \Sigma (\SGr_{n,np})) \bigotimes_{E_{*,*}}  E^{*,*}( \Sigma \mathds{P}^{np})\right) \cong \lim_p \left(\frac{E^{*,*}[c_2,..,c_n]}{I'} \otimes_{E_{*,*}} \frac{E^{*,*}[c_1]}{c_1^{np+1}} \right)\]
We are implicitly using the fact that the map $E^{*,*}(\Gr_{n,np}) \to E^{*,*}(\Gr_{1,np})$, induced by the canonical inclusion takes $c_1(\Taut_{n,np})$ to $c_1(\Taut_{1,np})$, and $c_i$ to $0$, for $i > 1$. So, we have 
\[\lim_p\left(E^{*,*} ( \Sigma (\SGr_{n,np})) \bigotimes_{E_{*,*}}  E^{*,*}( \Sigma \mathds{P}^{np})\right) \cong \lim_p\left(\frac{E^{*,*}[c_1,c_2,..,c_n]}{<c_1^{np+1}, I'>}\right)\]
\end{proof}
\begin{remark}\label{bs1.5}
    Again, looking at the limiting case \[\lim_p  \left( \frac{E^{*,*}[c_2,..,c_n]}{Im(\otimes c_1)} \right ) \cong \lim_p \left (\frac{E^{*,*}[c_2,..,c_n]}{I'} \right ) \cong E^{*,*}\llbracket c_2,..,c_n\rrbracket\]
    So, we have
    \[\lim_p \left (\frac{E^{*,*}[c_1,c_2,..,c_n]}{<c_1^{np+1}, I'>} \right) \cong E^{*,*}\llbracket c_1,c_2,..,c_n \rrbracket\]\\
\end{remark}
\begin{remark}
    The proof of \cref{bs2} was incorrect in the previous version. The author is very grateful to Matthias Wendt for finding the mistake, as well as providing an idea to fix it.
\end{remark}
\begin{lemma}\label[lemma]{bs2}
  Using the notation in \cref{bs1}, $Im(\otimes c_1)$ is finitely generated free as an $E^{*,*}$ module.
\end{lemma}
\begin{proof}
    Following the notation from \cite[section 6.2]{motlandweber}, we have \[E^{*,*}(\Gr_{n,np}) \cong E^{*,*} \otimes_{\Z} \frac{\Z[x_1,x_2,..,x_{np-p}]}{(s_{p+1},...,s_{np})}\] Consider the following universal short exact sequence on $\Gr(n-r,n)$. 
\[\begin{tikzcd}[column sep=small]
	0 & {\mathcal{K}_{n-r}} & {\mathcal{O}^n} & {\mathcal{Q}_r} & 0
	\arrow[from=1-1, to=1-2]
	\arrow[from=1-2, to=1-3]
	\arrow[from=1-3, to=1-4]
	\arrow[from=1-4, to=1-5]
\end{tikzcd}\]\\
Here $\mathcal{K}_{n-r}$ is the tautological subbundle, and $\mathcal{Q}_r$ is the tautological quotient bundle. Let $x_i$ be the $i$th Chern class of $\mathcal{K}_{n-r}$. In our notation, $c_i$ is the $i$th Chern class of $\mathcal{Q}_r$. Similar to \cite[page 183]{3264} or \cite{motlandweber}, $s_i$'s satisfy
\[1 + \Sigma s_i t^i = (1 + x_1 t+..+x_{n-d} t^{n-d})^{-1} \in \Z[x_1,x_2,..,x_{np-p}]\llbracket t \rrbracket^{\times}\]
Then as a finitely generated $\Z$ module, a basis for $\frac{\Z[x_1,x_2,..,x_{n-d}]}{(s_{d+1},...,s_{n})}$ can be given by the following:\\
For every $n$, such that $n-d \geq a_1 \geq a_2 \geq a_d \geq 0$, choose a tuple $a = (a_1,a_2,..,a_d)$. 
$$(\underline{a}) = det \begin{pmatrix}x_{a_1} & x_{a_1 +1} & ..... & x_{a_1 +d-1}\\
x_{a_2-1} & x_{a_2} & ..... & x_{a_2 +d-2}\\
..... & .....& .....& .....\\
x_{a_d -d +1} & ..... & ..... & x_{a_d}\\
\end{pmatrix}$$
Here $x_i = 0$, for $i < 0$ or $i > n-d$, and $x_0 = 1$ . These polynomials $(\underline{a})$ form a basis.
Now, we choose a particular $d$-tuple $a = (1,0,0,..,0)$. Then 
$$(\underline{a}) = det \begin{pmatrix}x_{1} & x_{2} & ..... & x_{d}\\
0 & 1 & ..... & x_{ d-2}\\
0& 0& .....& .....\\
0 & 0 & 0& 1\\
\end{pmatrix}$$
So, $x_1$ is a basis element of the basis described above. We have that, $c_1 = -x_1$ (Using Cartan formula for Chern classes), is a basis of the finitely generated free $E^{*,*}$ module $\frac{E^{*,*}[c_1,..,c_n]}{I}$. Therefore, the map denoted as $\otimes c_1$ does not take any non-zero element of $E^{*,*}$ to $0$.\\[1.0000ex]
As $\frac{\Z[x_1,x_2,..,x_{np-p}]}{(s_{p+1},...,s_{np})}$ is a finitely generated free Abelian group, so is the ideal $Im(\otimes (-x_1))$ in it. By tensoring with $E^{*,*}$, we get $Im (\otimes c_1)$ is finitely generated free as claimed.
\end{proof}
\subsection{A closer look at the map $\Psi \wedge E$} \label[subsection]{Eiso}
Let's consider the map $\Phi: \MSL \to \MGL$. From our previous computation the induced map \[\Phi^*: E^{*,*}(\MGL) \cong E^{*,*} \llbracket c_1,c_2,... \rrbracket \to  E^{*,*}(\MSL) \cong E^{*,*} \llbracket c_2,c_3,.. \rrbracket\] takes $c_1$ to $0$, and $c_i$ to $c_i$ for $i >1$. \\[1.5ex]
We have the map $\varphi: \Sigma ^{\infty-2, \infty-1} \mathds{P^{\infty}} \to \MGL$. The induced map \[\varphi^*: E^{*,*}(\MGL) \cong E^{*,*} \llbracket c_1,c_2,. \rrbracket \to  E^{*,*}(\Sigma ^{\infty-2, \infty-1} \mathds{P^{\infty}}) \cong E^{*,*} \llbracket c_1 \rrbracket\] takes $c_1$ to $c_1$, and $c_i$ to $0$, for $i >1$.\\[1.5ex]
We now consider the induced map $\mu^*$ induced from the ring spectrum map $\mu: \MGL \wedge \MGL \to \MGL$. \[\mu^*: E^{*,*}(\MGL) \cong E^{*,*}\llbracket c_1, c_2,..,c_n,.. \rrbracket \to E^{*,*}(\MGL \wedge \MGL) \cong E^{*,*}\llbracket c'_1,c''_1, c'_2,c''_2,..,c'_n,c''_n.. \rrbracket\]
We have $\mu^* (c_i) = \Sigma_{(m+n = i)} c'_m c''_n$.\\[1.5ex] Therefore $\Psi^* := (\Phi \wedge \varphi)^* \mu^* : E^{*,*}(\MGL) \to E^{*,*}(\MSL \wedge \Sigma^{\infty-2, \infty-1} \mathds{P^{\infty}})$  takes $c_1$ to $c'_1$, $c_2$ to $c''_2$, and $c_n$ to $c'_1 c''_{n-1} + c''_n$. \\[1.000ex]
Here $E^{*,*}(\MSL \wedge \Sigma^{\infty-2, \infty-1} \mathds{P^{\infty}}) \cong E^{*,*} \llbracket c'_1,c''_2,c''_3,...\rrbracket$. As the map $\Psi^*$ is an $E^{*,*}$ module map, it is an isomorphism.\\
As we would like to have $\Psi \wedge E$ to be an equivalence, our interest is in looking at the dual map $$\Psi_{*}:E_{*,*}(\MSL \wedge \Sigma^{\infty-2, \infty-1} \mathds{P^{\infty}}) \to E_{*,*}(\MGL) $$ $E_{*,*}(\MSL \wedge \Sigma^{\infty-2, \infty-1} \mathds{P^{\infty}})$ is the filtered colimit of  $E_{*,*}$ duals of \[E^{*,*}(\SGr_{p,np}) \otimes_{E_{*,*}} E^{*,*}(\mathds{P}^{n'})\] 
$E^{*,*}(\SGr_{p,np})$, and $E^{*,*}(\mathds{P}^{n'})$ are flat, and finitely generated free $E^{*,*}$-modules respectively. Similarly $E_{*,*}(\MGL)$ is colimit of duals of finitely generated free modules of form $E^{*,*}(\Gr_{p, pn})$. The cohomology isomorphism $\Psi^*$ is induced by these corresponding maps at the finite stages. As the correspondimg $\nlim^1$ terms vanish, we take colimit over the induced maps between the canonically isomorphic duals. Using the fact that $E^{*,*}(\mathds{P}^{n'})$ is finitely generated free, we get the homology isomorphism we want.\\
We can summarize this subsection in the form of the following result.
\begin{lemma}\label{CellularMSLsmashEequivalence}
 Let $S$ be a Noetherian scheme of finite Krull dimension and $E \in \SH(S)$ be an oriented cohomology theory with additive formal group law. Then, the map $\Psi$ induces an isomorphism of the bigraded stable homotopy groups
 \[\pi_{p,q}(\MSL \wedge \Sigma^{\infty-2, \infty-1} \mathds{P^{\infty}} \wedge E) \xrightarrow[\cong]{\pi_{p,q}(\Psi \wedge E)} \pi_{p,q}(\MGL \wedge E)\]    
\end{lemma}
In particular for any cellular $E$, \cref{CellularMSLsmashEequivalence} implies that $\Psi \wedge E$ is an isomorphism in $\SH(S)$. As the map $\Psi \wedge E$ induces isomorphism on the global section of the homotopy sheaves, using the following \cref{checkStableEquivalence} by Morel, we get \cref{MSLsmashEequivalence}.
\begin{theorem}\cite[Example 2.3, Proposition 2.8]{MorelBook}\label{checkStableEquivalence}
   Let $k$ be a perfect field, and $f: F \to G$ be a morphism of strictly $\A^1-$invariant Nisnevich sheaves on $\text{Sm}_k$. Then $f$ is an isomorphism iff for every finitely generated field extension $k \hookrightarrow l$, the induced map $F(l) \to G(l)$ is an isomorphism.
\end{theorem}
\begin{theorem}\label{MSLsmashEequivalence}
 Let $k$ be a perfect field, and $E \in \SH(k)$ be an oriented cohomology theory with additive formal group law. Then, the following map is an equivalence in $\SH(k)$
 \[\MSL \wedge \Sigma^{\infty-2, \infty-1} \mathds{P^{\infty}} \wedge E \xrightarrow[\cong]{\Psi \wedge E} \MGL \wedge E\]    
\end{theorem}
\begin{remark}
    The spectra $\text{MSU} \wedge \Sigma^{-2}\Sigma^{\infty} \mathbb{C P}^{\infty}$, and $ \text{MU}$ are both connective. In topology, due to Hurewicz theorem showing the map in \cref{msuCFIntro} induces an isomorphism on singular homology is enough to prove that the map is a stable homotopy equivalence.\\[1.00000ex]
    In $\A^1$-homotopy theory, we have Bachmann's Hurewicz theorem \cite[Theorem 16]{BachmannDUKE}, which is applicable on some selected cases. We can get that $\Psi$ is an equivalence over prime fields with odd characteristic using this. For more detail on this, and an Adams spectral sequence based approach showing $\Psi$ is an equivalence over characteristic $0$ fields, see the appendix of the author's thesis\cite[Appendix A]{MyThesis}.
\end{remark}
\subsection{The map $\Psi$ itself is an equivalence} \label[subsection]{finalproof}
We first show that both $\Psi [1/\eta]$, and $\Psi^{\wedge}_{\eta}$ are equivalences. Then $\eta$-fracture square \cref{etaFractureSquare} is used to finish the proof.
\begin{lemma} \cite[Lemma 3.9]{annals_stable}
    For every motivic spectrum $E$, and $f \in \pi_{s,t}(\mathds{1})$, we have the following pullback square
    \begin{equation} \label{etaFractureSquare}  
\begin{tikzcd}
	E && {E^{\wedge}_f} \\
	\\
	{E [1/f]} && {E^{\wedge}_f [1/f]}
	\arrow[from=1-1, to=1-3]
	\arrow[from=1-1, to=3-1]
	\arrow[from=3-1, to=3-3]
	\arrow[from=1-3, to=3-3]
	\arrow["\ulcorner"{anchor=center, pos=0.125, rotate=180}, draw=none, from=3-3, to=1-1]
\end{tikzcd}
 \end{equation}
\end{lemma}
\subsubsection{$\Psi [1/\eta]$ is an equivalence of contractible spectra}
We have discussed in \cref{etaMGL}, the unit map of $\MGL$ factors through $\mathds{1}/ \eta$, as $\MGL$ is the universal oriented cohomology theory, we get that inverting $\eta$ degenerates any oriented cohomology theory into a trivial theory. In particular, we have
\[\MGL [1/\eta] \cong *\]
It is enough to show that $\MSL \wedge \Sigma^{-2,-1} \Sigma^{\infty} \mathds{P}^{\infty} [1/\eta] \cong *$. 
\begin{proposition}\label{LevineIdeaForGeneralBase}
 Let $k$ be any field. In $\SH(k)$, we have the following equivalence \[(\MSL \wedge \Sigma^{-2,-1} \Sigma^{\infty} \mathds{P}^{\infty})[1/\eta] \cong *\]   
\end{proposition}
\begin{proof}
Let's consider the canonical closed embedding $\mathbb{P}^1 \xrightarrow{i}\mathbb{P}^{2n}$. The normal bundle of this has the form $\mathcal{N}(i) \cong \mathcal{O}(1)^{\oplus (2n-1)}_{\mathbb{P}^1}$. Its open complement is a vector bundle over $\mathbb{P}^{2n-2}$, whose zero section gives an $\mathds{A}^1$-equivalence $\mathbb{P}^{2n-2} \cong \mathbb{P}^{2n} \backslash i(\mathbb{P}^{1})$. \\[1.25ex]
In particular we have the cofiber sequence in $\mathcal{H}_{*}(k)$ 
\[\mathbb{P}^{2n-2} \hookrightarrow \mathbb{P}^{2n} \to Th_{\mathbb{P}^{1}}( \mathcal{O}(1)^{\oplus (2n-1)}) \cong \underbrace{Th_{\mathbb{P}^{1}}( \mathcal{O}(1)) \wedge Th_{\mathbb{P}^{1}}( \mathcal{O}(1)) \wedge .... \wedge Th_{\mathbb{P}^{1}}( \mathcal{O}(1)) }_{(2n-1) \text{times}} \]
Following \cite[Proposition 2.2]{thetacharacteristic}, In $\SH(k)$, for some field $k$, $Th_{X}( V) \cong Th_{X}( V^{\vee})$. Here $X$ is quasi-projective, and $V^{\vee}$ is the dual bundle of $V$. Now, we get the following cofiber sequence in $\SH(k)$ 
\[\Sigma^{\infty}\mathbb{P}^{2n-2} \hookrightarrow \Sigma^{\infty}\mathbb{P}^{2n} \xhookrightarrow{\cong} \underbrace{\Sigma^{\infty} Th_{\mathbb{P}^{1}}( \mathcal{O}(-1)) \wedge \Sigma^{\infty} Th_{\mathbb{P}^{1}}( \mathcal{O}(-1)) \wedge .... \wedge \Sigma^{\infty} Th_{\mathbb{P}^{1}}( \mathcal{O}(-1)) }_{(2n-1) \text{times}}\]
Using the cofiber sequence $\mathcal{O}(-1) \backslash \mathbb{P}^{1} \cong \mathds{A}^2 \backslash \{0\} \xrightarrow{\eta} \mathbb{P}^{1} \to Th_{\mathbb{P}^{1}}( \mathcal{O}(-1) \cong \mathbb{P}^{2}$, we get the following cofiber sequence in $\SH(k)$ \begin{equation} \label{ProjectiveEmbedding}
    \Sigma^{\infty}\mathbb{P}^{2n-2} \hookrightarrow \Sigma^{\infty}\mathbb{P}^{2n} \to  \underbrace{\Sigma^{\infty} \mathbb{P}^{2}  \wedge \Sigma^{\infty} \mathbb{P}^{2} \wedge .... \wedge \Sigma^{\infty} \mathbb{P}^{2} }_{(2n-1) \text{times}}
\end{equation}
As $\Sigma^{-2,-1}\Sigma \mathbb{P}^{2} \cong \mathds{1}/\eta$, $\Sigma^{-2,-1}\Sigma \mathbb{P}^{2} \cong *$ in $\SH(k)[1/\eta]$. Using \cref{ProjectiveEmbedding}, by induction on $n$, $\Sigma \mathbb{P}^{2n} \cong *$. So, we get $\Sigma \mathbb{P}^{\infty} \cong *$ in $\SH(k)[1/\eta]$.
\end{proof}
\begin{remark}
    Haution in \cite[Proposition 6.11 (ii)]{haution_2023} has shown that for any Noetherian base scheme $S$ of finite Krull dimension $\Sigma ^{\infty} Th_{\mathds{P}^{\infty}}(\mathcal{O}(n)) \cong * $ in $\SH(S)[1/\eta]$, whenever $n$ is odd. In particular we have $\Sigma^{\infty} (\mathds{P}^{\infty}, \infty) \cong \Sigma ^{\infty} Th_{\mathds{P}^{\infty}}(\mathcal{O}(-1))$ is contractible in $\SH(S)[1/\eta]$.\\[1.00000ex] 
    In fact, the assertion of \cref{LevineIdeaForGeneralBase} is true over any Noetherian base scheme $S$ of finite Krull dimension.
\end{remark}
\subsubsection{$\Psi^{\wedge}_{\eta}$ is an equivalence}

We have a straightforward way of showing $\Psi^{\wedge}_{\eta}$ is an equivalence by comparing completion with respect to $\eta$, $\HZ$ and $\MGL$. This works over any perfect field as outlined below. The author learned this from Oliver R\"ondigs.
\begin{proposition}\label[proposition]{PsiIntegralActual}
  We have the following equivalence over any perfect field \[(\MSL \wedge \Sigma^{-2,-1} \Sigma^{\infty} \mathds{P}^{\infty})^{\wedge}_{\eta} \xrightarrow{\Psi^{\wedge}_{\eta}} \MGL\]
\end{proposition}
\begin{proof} 
From \cite[Theorem 7.4]{Hoyoismalgmot}, taking the base to be a field $k$, the unit of $\HZ$, $\mathds{1} \to \HZ$ induces an isomorphism $(\mathds{1}/{\eta})_{\leq 0} \xrightarrow{\cong} \HZ_{\leq 0} $. Morel \cite{MorelPiZero} computed $\underline{\pi}_0 (\mathds{1})$ over a field. He showed that $\underline{\pi}_0 (\mathds{1}) \cong \underline{K}^{MW}$.\\[1.12ex]
Here $\underline{K}^{MW}$ denotes Milnor-Witt $K$-theory. Using the relation $\underline{K}^{MW}/\eta \cong \underline{K}^{M}$, as a corollary we get, $\HZ$ is connective and \[\underline{\pi}_0 (\HZ) \cong \underline{\pi}_0 (\mathds{1}/\eta) \cong \underline{K}^{M}\]
From \cref{etaMGL}, over a general base scheme, we have $(\mathds{1}/{\eta})_{\leq 0} \xrightarrow{\cong} \MGL_{\leq 0} $. In particular, over a field 
\[\underline{\pi}_0 (\MGL) \cong \underline{\pi}_0 (\mathds{1}/\eta) \cong \underline{K}^{M}\]
Here $\underline{K}^{M} \in \text{CAlg}(\SH(k)^{\heartsuit})$ denotes Milnor $K$-theory. As $\underline{K}^{M} \cong \underline{\pi}_0 (\mathds{1}/\eta)$, it is an idempotent using \cite[Remark 7.2.4]{mantovani}. In other words, $\underline{K}^{M} \otimes \underline{K}^{M} \to \underline{K}^{M} \in \SH(k)^{\heartsuit}$ is an isomorphism. Here we are considering $\SH(k)$ with the homotopy $t$-structure (see \cref{definitionHomotopyTstructure}).\\[1.2ex]
From \cref{completion} and \cref{completion2}, for $X$, a connective spectrum in $\SH(k)$,
\begin{equation}\label{CompareCompletion}
 X^{\wedge}_{\MGL} \cong X^{\wedge}_{\eta} \cong X^{\wedge}_{\HZ}   
\end{equation}
$\MGL$ is trivially $\MGL$ complete, and therefore $\eta$ complete.\\
Now, using \cref{MSLsmashEequivalence}, over a perfect field, we have \[\MSL \wedge \Sigma^{-2,-1} \Sigma^{\infty} \mathds{P}^{\infty} \wedge \HZ \xrightarrow[\cong]{\Psi \wedge \HZ} \MGL \wedge \HZ\]
Therefore we have $\Psi^{\wedge}_{\HZ}$ is an equivalence. Using \cref{CompareCompletion}, $\Psi^{\wedge}_{\eta}$ is an equivalence over perfect fields. 
\end{proof}
\subsubsection{$\Psi$ is an equivalence}
\begin{theorem} \label[theorem]{onlyresult}
    Let $k$ be a perfect field. The map \[\Psi: \MSL \wedge \Sigma^{-2,-1} \Sigma^{\infty} \mathds{P}^{\infty} \to \MGL\] is an equivalence in $\SH(k)$.
\end{theorem}
\begin{proof}
    Combining the equivalence in \cref{PsiIntegralActual}, \cref{LevineIdeaForGeneralBase} and $\eta$- fracture square \cref{etaFractureSquare} for $\MSL \wedge \Sigma^{-2,-1} \Sigma^{\infty} \mathds{P}^{\infty}$, We have the following pullback square 
\[\begin{tikzcd}[column sep=tiny]
	{\MSL \wedge \Sigma^{-2,-1} \Sigma^{\infty} \mathds{P}^{\infty}} && {} && {(\MSL \wedge \Sigma^{-2,-1} \Sigma^{\infty} \mathds{P}^{\infty})^{\wedge}_{\eta}} & \MGL \\
	\\
	{(\MSL \wedge \Sigma^{-2,-1} \Sigma^{\infty} \mathds{P}^{\infty})[1/\eta]} &&&& {(\MSL \wedge \Sigma^{-2,-1} (\Sigma^{\infty} \mathds{P}^{\infty})^{\wedge}_{\eta})[1/\eta]} & {\MGL[1/\eta]} \\
	{*} &&&&& {*}
	\arrow[from=1-1, to=1-5]
	\arrow[from=1-1, to=3-1]
	\arrow[from=3-1, to=3-5]
	\arrow[from=1-5, to=3-5]
	\arrow["\cong", from=1-5, to=1-6]
	\arrow["\cong", from=3-5, to=3-6]
	\arrow["\cong", from=3-1, to=4-1]
	\arrow["\cong", from=3-6, to=4-6]
	\arrow["\ulcorner"{anchor=center, pos=0.125, rotate=180}, draw=none, from=3-5, to=1-3]
\end{tikzcd}\]
We immediately get that the top horizontal map \[\MSL \wedge \Sigma^{-2,-1} \Sigma^{\infty} \mathds{P}^{\infty} \to (\MSL \wedge \Sigma^{-2,-1} \Sigma^{\infty} \mathds{P}^{\infty})^{\wedge}_{\eta}\] is an equivalence in $\SH(k)$ and we are done.
\end{proof}
Now we show a generalization of \cref{onlyresult} as promised in the introduction.
We use the following less general and simplified form of \cite[B.2]{normsinmotivichomotopy}. For $E \in \SH(\Z)$, let's denote $E_{\Z/p\Z} \in \SH(\Z/p\Z)$ be the pullback of $E$ along a closed point $(p) \in \text{Spec}(\Z)$, for some prime $p$, and $E_{\Q} \in \SH(\Q)$ be the  pullback of $E$ along the generic point $(0) \in \text{Spec}(\Z)$.
\begin{lemma}\label{detecting0inSHZ}
  Let $E \in \SH(\Z)$. Then $E \cong 0 \in \SH(\Z)$ iff $E_{\Z/p\Z} \cong 0$ in $\SH(\Z/p\Z)$, and $E_{\Q} \cong 0$ in $\SH(\Q)$.
\end{lemma}
\begin{proof}
    As we can check contractibility of an object stalkwise, it is enough to consider the spectrum in the category $\SH(\Z_{(p)})$. We have the $(0) \xhookrightarrow{j} \text{Spec}(\Z_{(p)}) \xhookleftarrow{i} (p)$, where $i$ is a closed immersion, and $j$ is the open immersion of its open complement. Using Morel-Voevodsky localization theorem \cite[p.114, Theorem 2.21]{MV99}, we have the cofiber sequence \[j_{\#}j^* E \to E \to i_*i^* E\]
    By assumption $i^* E \cong 0$, and $j^* E \cong 0 $. 
\end{proof}
\begin{corollary} \label{onlyresultZ}
   Let $S$ be a Noetherian scheme of finite Krull dimension.The map $\Psi: \MSL \wedge \Sigma^{-2,-1} \Sigma^{\infty} \mathds{P}^{\infty} \to \MGL$ is an equivalence in $\SH(S)$. 
\end{corollary}
\begin{proof}
   Consider a scheme $S$ and $s$ is a point on $S$. By abuse of notation, let $E_s$ denote the pullback of $E \in \SH(S)$ to $\SH(\kappa(s))$.\\[1.00000ex] 
   In particular we can consider the map $\Psi: \MSL \wedge \Sigma^{-2,-1} \Sigma^{\infty} \mathds{P}^{\infty} \to \MGL$ over $Spec(\Z)$. To show that the homotopy cofiber of the map, $\hocofib(\Psi)$ is $0 \in \SH(\Z)$, using \cref{detecting0inSHZ}, it is enough to show that $\hocofib(\Psi)_{\Q}$ is $0 \in \SH(\Q)$, and $\hocofib(\Psi)_{\Z/p\Z}$ is $0 \in \SH(\Z/p\Z)$. Using \cref{onlyresult}, we know that $\hocofib(\Psi)_{k}$ is $0 \in \SH(k)$ for any perfect field $k$. So, we have that the map $\Psi$ is an equivalence when constructed in $\SH(\Z)$.\\[1.000000ex]
 As Grassmannians behave well with base change, we have $\MSL$, $\MGL$, $\Sigma^{\infty} \mathds{P}^{\infty}$ and the map $\Psi$ are compatible with respect to pullback along $S \to \text{Spec}(\Z)$, and hence we are done.
\end{proof}
\section{Some applications} \label{Some applications}
\subsection{Stable homotopy groups of $\MSL$}
\cref{onlyresultZ} gives us the following filtration:
\begin{equation}\label{skeletal}
 \begin{tikzcd}
	{\MSL } && {\MSL \wedge \Sigma^{-2,-1} \Sigma^{\infty} \mathds{P}^2} && {\MSL \wedge \Sigma^{-2,-1} \Sigma^{\infty} \mathds{P}^3} & {...} & \MGL \\
	&& {\Sigma^{2,1} \MSL} && {\Sigma^{4,2} \MSL}
	\arrow[hook, from=1-1, to=1-3]
	\arrow[hook, from=1-3, to=1-5]
	\arrow[hook, from=1-5, to=1-6]
	\arrow[from=1-6, to=1-7]
	\arrow[from=1-3, to=2-3]
	\arrow[from=1-5, to=2-5]
\end{tikzcd}   
\end{equation}
In the second row, we have shown the homotopy cofiber of the map, for $k \in \N$ \[\MSL \wedge \Sigma^{-2,-1} \Sigma^{\infty} \mathds{P}^k \to \MSL \wedge \Sigma^{-2,-1} \Sigma^{\infty} \mathds{P}^{k+1} \]
In particular, we get a first quadrant spectral sequence of the following form
\begin{equation} \label{FSS}
  E_1^{p,q,n} = \pi_{n+p+q,n}(\Sigma^{2p,p} \MSL) \Longrightarrow \pi_{n+p+q,n}(\MGL)  
\end{equation}
with differentials of the following shape 
\[d_r: E_r^{p,q,n} \to E_r^{p-r,q-1,n} \]
To determine the $d_1$-differentials, we need to determine the maps
\[\pi_{n+p+q,n} (\Sigma^{2p,p} \MSL) \to \pi_{n+(p+q-1),n} (\Sigma^{2p-2,p-1} \MSL)\]
These maps are induced by the maps of the form 
\[\begin{tikzcd}[column sep=tiny,row sep=scriptsize]
	{\pi_{n+p+q,n}(\MSL \wedge \Sigma^{-2,-1} \Sigma^{\infty} (\mathds{P}^{p+1}/\mathds{P}^{p}) )} &&&&&& {\pi_{n+(p+q-1),n}(\MSL \wedge \Sigma^{-2,-1} \Sigma^{\infty} \mathds{P}^{p})} \\
	\\
	&&&&&& {\pi_{n+(p+q-1),n}(\MSL \wedge \Sigma^{-2,-1} \Sigma^{\infty} (\mathds{P}^{p}/\mathds{P}^{p-1}) )}
	\arrow[from=1-1, to=1-7]
	\arrow[from=1-7, to=3-7]
\end{tikzcd}\]
\begin{remark}

Let's look at the cellular chain complex of $\R\mathds{P}^n$. There is one cell in each dimension between $0$, and $n$. The periodicity we notice here is the main motivation for us to guess how some of the $d_1$-differential maps might look. Following \cite[Example 2.42]{Hatcher}, the cellular complex looks like 
\[\begin{tikzcd}[column sep=small,row sep=scriptsize]
	0 & {C_0(\mathbb{R}\mathds{P}^n)} & {C_1(\mathbb{R}\mathds{P}^n)} & {C_2(\mathbb{R}\mathds{P}^n)} & {...} & {C_n(\mathbb{R}\mathds{P}^n)} & 0 \\
	0 & \Z & \Z & \Z & {...} & \Z & 0 \\
	& {}
	\arrow[from=1-3, to=1-2]
	\arrow[from=1-4, to=1-3]
	\arrow[from=1-5, to=1-4]
	\arrow[from=1-6, to=1-5]
	\arrow[from=2-6, to=2-5]
	\arrow[from=1-7, to=1-6]
	\arrow[from=2-7, to=2-6]
	\arrow[from=2-5, to=2-4]
	\arrow[from=2-4, to=2-3]
	\arrow[from=2-3, to=2-2]
	\arrow[from=2-2, to=2-1]
	\arrow[from=1-2, to=1-1]
	\arrow["{=}", from=1-2, to=2-2]
	\arrow["{=}", from=1-3, to=2-3]
	\arrow["{=}", from=1-4, to=2-4]
	\arrow["{=}", from=1-6, to=2-6]
\end{tikzcd}\]
The differential is determined by the following map 
\[\begin{tikzcd}[column sep = tiny]
	{C_{r-1}} & {H_{r-1}(\mathbb{R}\mathds{P}^{r-1}, \mathbb{R}\mathds{P}^{r-2})} && {H_{r-1}(\mathbb{R}\mathds{P}^{r-1})} & {H_{r}(\mathbb{R}\mathds{P}^{r}, \mathbb{R}\mathds{P}^{r-1})} & {C_{r}} \\
	& {H_{r-1}(D^{r-1}/S^{r-2},*)} && {H_{r-1}(S^{r-1})} & {H_{r}(D^{r}, S^{r-1})}
	\arrow["{\pi_*}", from=2-4, to=1-4]
	\arrow["\cong", from=2-5, to=1-5]
	\arrow["\cong", from=2-2, to=1-2]
	\arrow["{=:}"', from=1-6, to=1-5]
	\arrow[from=1-5, to=1-4]
	\arrow[from=1-4, to=1-2]
	\arrow[from=2-4, to=2-2]
	\arrow["{:=}"', from=1-2, to=1-1]
	\arrow[from=2-5, to=2-4]
	\arrow[from=1-4, to=2-2]
\end{tikzcd}\]

Here $\pi: S^{r-1} \to \mathbb{R}\mathds{P}^{r-1}$ is the double cover. The map $S^{r-1} \to D^{r-1}/S^{r-2}$ factors as follows
\[\begin{tikzcd}[sep=scriptsize]
	{S^{r-1}} && {\mathbb{R}\mathds{P}^{r-1}} && {D^{r-1}/S^{r-2}} \\
	&& {S^{r-1}/S^{r-2} \cong S^{r-1} \vee S^{r-1}}
	\arrow["\pi", from=1-1, to=1-3]
	\arrow[from=1-3, to=1-5]
	\arrow["\alpha"', from=1-1, to=2-3]
	\arrow[from=2-3, to=1-5]
\end{tikzcd}\]
One component of $\alpha$ is identity, and the other component is the antipodal map. The degree of the antipodal map on $S^n$ is $(-1)^{n+1}$. So, the cellular complex looks like 
\[0 \xleftarrow{0} \Z \xleftarrow{2} \Z \xleftarrow{0} \Z \xleftarrow{2} \Z... \xleftarrow{} \Z \xleftarrow{} 0\]
\end{remark}
Now onwards, we work in $\SH(k)$, where $k$ is a field of characteristic other than $2$. 
\begin{remark}
    We have two fiber sequences \cite[p.11]{AsokFaselEta}
    \[\text{SL}_{n}/\text{SL}_{n-1} \to \text{BSL}_{n-1} \to \text{BSL}_{n}\]
    and 
    \[\text{SL}_{n-1} \to \text{SL}_{n} \to \text{SL}_{n}/\text{SL}_{n-1} \]
    We consider the following composite
    \[\Omega^1_s \text{SL}_{n}/\text{SL}_{n-1} \to \Omega^1_s \text{BSL}_{n-1} \cong \text{SL}_{n-1} \to \text{SL}_{n-1}/\text{SL}_{n-2} \]
    Applying the functor $\underline{\pi}_{n-2}^{\mathbb{A}^1}$, we get 
    \[\underline{K}^{MW}_{n}\cong \underline{\pi}_{n-2}^{\mathbb{A}^1}(\Omega^1_s \text{SL}_{n}/\text{SL}_{n-1}) \xrightarrow{\delta_{n-1}} \underline{\pi}_{n-2}^{\mathbb{A}^1}(\text{SL}_{n-1}) \xrightarrow{q_{n-2}} \underline{\pi}_{n-2}^{\mathbb{A}^1}(\text{SL}_{n-1}/\text{SL}_{n-2}) \cong \underline{K}^{MW}_{n-1} \]
    Asok-Fasel explicitly described the composition $q_{n-2} \circ \delta_{n-1} $.
\begin{lemma}\label{AFeTA} \cite[Lemma 3.5]{AsokFaselEta}
 The composition $q_{n-2} \circ \delta_{n-1}$ is $\eta$, when $n$ is odd, and $0$ for even $n$.   
\end{lemma}
\end{remark}
\begin{remark}
Let's define $\mathcal{Q}_{2n-1}$ to be the quadric hypersurface in $\mathds{A}^n \times \mathds{A}^n$, defined by the equation $\sum_{i=1}^n x_iy_i =1$. We consider the $\mathds{G}_m$-action on $\mathcal{Q}_{2n-1}$, given by $\lambda(a,b) = (\lambda a, \lambda^{-1} b)$. The quotient scheme $\mathcal{Q}_{2n-1}/\mathds{G}_m$ is weakly equivalent to $\mathds{P}^{n-1}$ \cite[p.1174]{AutomorphismP2}. R\"ondigs has constructed certain maps $\psi_n: \Sigma^{1,1} \mathcal{Q}_{2n-1}/\mathds{G}_m \to \SL_n$, and proved that the following diagram commutes \cite[4.2]{AutomorphismP2}
\[\begin{tikzcd}
	{\Sigma^{1,1}\mathds{P}^{n-1}} & {\Sigma^{1,1} \mathcal{Q}_{2n-1}/\mathds{G}_m} && {\SL_n} \\
	{\Sigma^{1,1}\mathds{P}^{n}} & {\Sigma^{1,1} \mathcal{Q}_{2n+1}/\mathds{G}_m} && {\SL_{n+1}}
	\arrow[hook, from=1-1, to=2-1]
	\arrow["\cong"', from=1-2, to=1-1]
	\arrow["{\psi_n}", from=1-2, to=1-4]
	\arrow[hook, from=1-2, to=2-2]
	\arrow[hook, from=1-4, to=2-4]
	\arrow["\cong", from=2-2, to=2-1]
	\arrow["{\psi_{n+1}}", from=2-2, to=2-4]
\end{tikzcd}\]
R\"ondigs also showed that the inclusion $\text{SL}_n \to \text{SL}_{n+1}$ sits in the following cofiber sequence 
\[\text{SL}_n \to \text{SL}_{n+1} \to \Sigma^{2n+1,n+1} (\text{SL}_n)_+ \]
Using this cofiber sequence R\"ondigs showed that in the following commutative diagram, where the top row is a (homotopy) cofiber, and the bottom row is a (homotopy) fiber sequence; the right most vertical map is a weak equivalence \cite[Corollary 4.3]{AutomorphismP2}.
\[\begin{tikzcd}
	{\Sigma^{1,1} \mathcal{Q}_{2n-1}/\mathds{G}_m} && {\Sigma^{1,1} \mathcal{Q}_{2n+1}/\mathds{G}_m} && {S^{2n+1,n+1}} \\
	{\text{SL}_n} && {\text{SL}_{n+1}} && {\text{SL}_{n+1}/\text{SL}_n}
	\arrow[from=1-1, to=1-3]
	\arrow["{\psi_n}", from=1-1, to=2-1]
	\arrow[from=1-3, to=1-5]
	\arrow["{\psi_{n+1}}"', from=1-3, to=2-3]
	\arrow["\cong", from=1-5, to=2-5]
	\arrow[from=2-1, to=2-3]
	\arrow[from=2-3, to=2-5]
\end{tikzcd}\]
We have the following commutative diagram 
\[\begin{tikzcd}[sep=small]
	{\underline{K}^{MW}_{n}} & {\underline{\pi}_{n-1}^{\mathbb{A}^1}(\text{SL}_{n}/\text{SL}_{n-1})} && {\underline{\pi}_{n-2}^{\mathbb{A}^1}(\text{SL}_{n-1})} && {\underline{\pi}_{n-2}^{\mathbb{A}^1}(\text{SL}_{n-1}/\text{SL}_{n-2})} & {\underline{K}^{MW}_{n-1}} \\
	\\
	& {\underline{\pi}_{n-1}^{\mathbb{A}^1}(\Sigma^{1,1} \mathbb{P}^{n-1}/\mathbb{P}^{n-2})} && {\underline{\pi}_{n-2}^{\mathbb{A}^1}(\mathbb{P}^{n-2})} && {\underline{\pi}_{n-2}^{\mathbb{A}^1}(\Sigma^{1,1}\mathbb{P}^{n-2}/\mathbb{P}^{n-3})}
	\arrow["\cong", from=3-2, to=1-2]
	\arrow["{\delta_{n-1}}", from=1-2, to=1-4]
	\arrow[from=3-2, to=3-4]
	\arrow[from=3-4, to=1-4]
	\arrow["{q_{n-2}}", from=1-4, to=1-6]
	\arrow[from=3-4, to=3-6]
	\arrow["\cong"', from=3-6, to=1-6]
	\arrow["\cong", from=1-1, to=1-2]
	\arrow["\cong", from=1-6, to=1-7]
\end{tikzcd}\]
The proof of \cite[Corollary 4.3]{AutomorphismP2} works the same if one stabilizes the two cofiber sequences 
\[\text{SL}_n \to \text{SL}_{n+1} \to \Sigma^{2n+1,n+1} (\text{SL}_n)_+\]
and 
\[\Sigma^{1,1} \mathcal{Q}_{2n-1}/\mathds{G}_m \to \Sigma^{1,1} \mathcal{Q}_{2n+1}/\mathds{G}_m \to S^{2n+1,n+1}\]
and smash both with $\MSL$.
\end{remark}
Following \cite[Remark 6.4.2]{MorelBook} For $n\geq 2$, and $i \geq 1$ 
\begin{equation}\label{Morel_stable_unstable}
  \underline{\pi}_{n}^{\mathbb{A^1}}(S_s^{n} \wedge \mathds{G}_m^i) \cong \underline{K}^{MW}_{i}  
\end{equation}
The $d_1$-differentials we are interested in are completely determined by the maps described by the vertical curved arrows in the following diagram after stabilization. The leftmost map is by construction $\eta$. 
\[\begin{tikzcd}
	{S^{3,2}} & {S^{5,3}} & {S^{7,4}} \\
	{\mathds{A}^2\backslash0} & {\mathds{A}^3\backslash0} & {\mathds{A}^4\backslash0} & {....} \\
	{\mathds{P}^1} & {\mathds{P}^2} & {\mathds{P}^3} & {.....} \\
	{\mathds{P}^1/\mathds{P}^0} & {\mathds{P}^2/\mathds{P}^1} & {\mathds{P}^3/\mathds{P}^2} & {....} \\
	{S^{2,1}} & {S^{4,2}} & {S^{6,3}}
	\arrow["\cong", from=1-1, to=2-1]
	\arrow["\cong", from=1-2, to=2-2]
	\arrow["\cong", from=1-3, to=2-3]
	\arrow[from=2-1, to=3-1]
	\arrow["\eta"', bend right, from=2-1, to=4-1]
	\arrow[from=2-2, to=3-2]
	\arrow[ bend left, from=2-2, to=4-2]
	\arrow[from=2-3, to=3-3]
	\arrow[ bend left, from=2-3, to=4-3]
	\arrow[hook, from=3-1, to=3-2]
	\arrow["{=}", from=3-1, to=4-1]
	\arrow[hook, from=3-2, to=3-3]
	\arrow[from=3-2, to=4-2]
	\arrow[hook, from=3-3, to=3-4]
	\arrow[from=3-3, to=4-3]
	\arrow["\cong", from=4-1, to=5-1]
	\arrow["\cong", from=4-2, to=5-2]
	\arrow["\cong", from=4-3, to=5-3]
\end{tikzcd}\]
From \cref{Morel_stable_unstable}, $[S^{2(n+1)-1,n+1}, S^{2n,n}]_{\mathcal{H}(k)_{*}} \cong [\Sigma^{\infty} S^{2(n+1)-1,n+1}, \Sigma^{\infty} S^{2n,n}]_{\SH(k)}$, for $n\geq 2$. We denote the pointed unstable motivic homotopy category over a field $k$ by $\mathcal{H}(k)_{*}$. Using \cref{AFeTA}, these maps are alternatively $\eta$, and $0$.\\[1.000000ex]
Using the previous identification, the $E_1$-page with some of the $d_1$-differentials is shown in the following picture \cref{EonePAGE}. Here, the $E^1$-page is shown for a fixed $n$. We have written $E^1_{p,q,n} = \pi_{n+p+q,n}(\Sigma^{2p,p} \MSL)$ as $\pi_{p+q}(\Sigma^{2p,p} \MSL)$ for notational convenience.\\[1.000000ex] 
\begin{equation} \label{EonePAGE}
 \begin{tikzpicture}
  \matrix (m) [matrix of math nodes,
    nodes in empty cells,nodes={minimum width=5ex,
    minimum height=5ex,outer sep=-5pt},
    column sep=1ex,row sep=1ex]{
                &      &     &     &  &\\
          p    &        &    &      0& ...&\\      
         2     &        & 0   &  \pi_0(\Sigma^{2,2}\MSL)    &... &\\
          1     &  0 &  \pi_0(\Sigma^{1,1}\MSL)  & \pi_1(\Sigma^{1,1}\MSL) &...  &\\
          0     &  \pi_0(\MSL)  & \pi_1(\MSL) &  \pi_2(\MSL)  &...  &\\
    \quad\strut &   0  &  1  &  2  & p+q &\strut \\
          };
          \draw[thick] (m-1-1.east) -- (m-6-1.east) ;
\draw[thick] (m-6-1.north) -- (m-6-6.north)  ;
\draw[-stealth] (m-4-3) -- (m-5-2) node[midway, left]{$\eta$};
\draw[-stealth] (m-4-4) -- (m-5-3) node[midway, left]{$\eta$};
\draw[-stealth] (m-4-5) -- (m-5-4) node[midway, left]{$\eta$};
\draw[-stealth] (m-3-4) -- (m-4-3) node[midway, left]{$0$};
\draw[-stealth] (m-3-5) -- (m-4-4) node[midway, left]{$0$};
\draw[-stealth] (m-2-5) -- (m-3-4) node[midway, left]{$\eta$};
\end{tikzpicture}   
\end{equation}

For some $E \in \SH(k)$, let's denote the global section of the homotopy module $\underline{\pi}_0(E)(\text{Spec}(k)) \cong \bigoplus_n \underline{\pi}_{n,n}(E)(\text{Spec}(k))$ as ${\pi}_0(E)$. As no higher differential goes to the $(0,0)$-th place, from \cref{EonePAGE}, we have that 
\begin{equation} \label{piZEROmsl}
 {\pi}_0(\MSL)/\eta \cong {\pi}_0(\MGL) \cong K^{M}_*(k)   
\end{equation}
\begin{remark} \label{MSLnorm}
    Let's consider $\MSL$ as a $T^2$-spectrum. Using a geometric argument like \cref{etaMGL}, it can be shown that the cofiber of the unit map $\mathds{1} \to \MSL$ lies the subcategory $\Sigma^{4,2}\SH(k)^{\text{veff}}$ \cite[Example 16.35]{normsinmotivichomotopy}. Here $(\SH(k))^{\text{veff}}$ denotes Spitzweck's very effective subcategory. It is the (non-triangulated) subcategory of $\SH(k)$ generated by $\{\Sigma^{\infty} X_+ : X \in \text{Sm}_k\}$ under colimits and extension. The idea is outlined below.\\[1.000000ex]
    The skeletal filtration of $\MSL$ starts with the unit map. The map from the $n$-th to the $(n+1)$-th skeleton looks like the following
    \begin{equation}\label{skeletalMSL}
     \Sigma^{4,2} \MSL_{2n} \to \MSL_{2(n+1)}   
    \end{equation}
    The map is induced by taking a sequential colimit of maps of Thom spaces of vector bundles of rank $2(n+1)$, where the bundle is restricted from the higher dimensional special linear grassmannian to the lower dimensional one. The cofiber of the map \cref{skeletalMSL} lies in $\Sigma^{4n,2n}\SH(k)^{\text{veff}}$.
     
    From the previous discussion, as the cofiber of the unit map of $\MSL$ lies in $\Sigma^{4n,2n}\SH(k)^{\text{veff}}$, it is $1$-connective. So, the unit map induces the following isomorphism. \[\underline{\pi}_0(\MSL) \cong \underline{\pi}_0(\mathds{1}) \cong \underline{K}^{MW}\]
    This is also explicitly shown in \cite{Yakerson_2021} over characteristic $0$ fields.
    As $K^{MW}_*(k)/\eta \cong K^{M}_*(k)$, \cref{piZEROmsl} matches with the well-known computation.
\end{remark}
For the following \cref{factorizationOfUnitOfkq}, and \cref{pi1msloliver}, we work over fields of characteristic other than $2$, and fields of characteristic $0$ respectively. Here $\text{KQ}$ denotes Hornbostel's Hermitian $K$-theory spectrum \cite[Section 3]{slicesofhermitianktheory}.
\begin{lemma}\label{factorizationOfUnitOfkq}
 Let $\textbf{kq}$ denote the very effective cover of $\text{KQ}$. The unit map of $\textbf{kq}$ factors through $\MSL$.
\end{lemma}
\begin{proof}
We have seen that $\MSL$ can be given an $\text{Sp}$-orientation (\cref{MSptoMSL}) $\psi_{\MSL}$, and Hermitian $K$-theory $\text{KQ}$ can be given an $\text{SL}$-orientation $\phi_\text{KQ}$ \cite{ananyevskiy2019sloriented}. From the universal properties of $\text{MSp}$, and $\MSL$, as described in \cref{MSpUniversal} and \cref{MSLuniversal}, we get the following. The commutativity of the diagram is explained below.
\begin{equation}\label{pi1MSLcomm}
\begin{tikzcd}[sep=scriptsize]
	{\mathds{1}} && {\text{MSp}} && \MSL \\
	\\
	&&&& {\text{KQ}}
	\arrow["{u_{\text{MSp}}}", from=1-1, to=1-3]
	\arrow["{\psi_{\MSL}}", from=1-3, to=1-5]
	\arrow["{\psi_{\text{KQ}}}", from=1-3, to=3-5]
	\arrow["{\phi_{\text{KQ}}}", from=1-5, to=3-5]
	\arrow["{u_{\text{KQ}}}"', from=1-1, to=3-5]
\end{tikzcd}
\end{equation}
In the previous diagram \cref{pi1MSLcomm}, for a commutative ring spectrum $E$, $u_E$, $\psi_E$, and $\phi_E$ denote the unit map, a fixed $\text{Sp}$-orientation, and a fixed $\SL$-orientation of $E$ respectively.\\[1.000ex]
As $\psi_{\MSL}$ is a monoidal map, we have the following factorization of the unit map of $\MSL$.
\[\begin{tikzcd}[sep=scriptsize]
	{\mathds{1}} && \MSL \\
	& {\text{MSP}}
	\arrow["{u_{\text{MSp}}}"', from=1-1, to=2-2]
	\arrow["{u_{\MSL}}", from=1-1, to=1-3]
	\arrow["{\psi_{\MSL}}"', from=2-2, to=1-3]
\end{tikzcd}\]
Similarly, as $\psi_{\text{KQ}}$ is a monoidal map, $\psi_{\text{KQ}}\circ u_{\text{MSp}} = u_{\text{KQ}}$.\\[1.0000ex]
We have uniqueness of a map from $\text{MSP}$, that determines a chosen $\text{Sp}$-orientation from \cref{MSpUniversal}. Using the uniqueness of $\psi_{\text{KQ}}$, we have $\phi_{\text{KQ}} \circ \psi_{\text{MSL}} = \psi_{\text{KQ}}$. So, the diagram \cref{pi1MSLcomm} commutes. In particular, we get that the unit map of $\text{KQ}$ factors through $\MSL$.
\begin{equation}\label{unitofKQ}
\begin{tikzcd}[sep=scriptsize]
	{\mathds{1}} && \MSL \\
	&& {\text{KQ}}
	\arrow["{u_{\MSL}}", from=1-1, to=1-3]
	\arrow["{u_{\text{KQ}}}"', from=1-1, to=2-3]
	\arrow["{\phi_{\text{KQ}}}", from=1-3, to=2-3]
\end{tikzcd}
\end{equation}
As $\mathds{1}$, and $\MSL$ are very effective spectra, \cref{unitofKQ} gives the following factorization of the unit map of $\textbf{kq}$.
\begin{equation}\label{unitofkqfinal}
\begin{tikzcd}[sep=scriptsize]
	{\mathds{1}} && \MSL \\
	&& {\textbf{kq}}
	\arrow["{u_{\MSL}}", from=1-1, to=1-3]
	\arrow["{u_{\textbf{kq}}}"', from=1-1, to=2-3]
	\arrow[from=1-3, to=2-3]
\end{tikzcd}   
\end{equation}
\end{proof}
The following \cref{pi1msloliver} is entirely based on Oliver R\"ondigs' ideas. The author is grateful to him for sharing this.\\[1.0000ex]
As before, for some $E \in \SH(k)$, we denote the global section of the homotopy module $\underline{\pi}_1(E)(\text{Spec}(k)) \cong \bigoplus_n \underline{\pi}_{n+1,n}(E)(\text{Spec}(k))$ as ${\pi}_1(E)$.
\begin{theorem}\label{pi1msloliver}
    Over a field of characteristic $0$, we have the following isomorphism
    \[\pi_{1}(\MSL) \cong \pi_{1}(\textbf{kq})\]
\end{theorem}
\begin{proof}
    From the discussion in \cref{MSLnorm}, using long exact sequence of stable homotopy groups, we get that the unit map $\mathds{1} \to \MSL$ induces the following surjection
\[\pi_1(\mathds{1}) \twoheadrightarrow \pi_1(\MSL)\]
R\"ondigs-\O{}stv\ae r-Spitzweck showed that the unit map $\mathds{1} \to \textbf{kq}$ induces the following surjection \cite[Theorem 2.2]{roendigs2019remarks}
\[\pi_1(\mathds{1}) \twoheadrightarrow \pi_1(\textbf{kq})\]
From \cref{piSurj}, we get that the map described in \cref{factorizationOfUnitOfkq}, $\pi_1(\MSL) \to \pi_1(\textbf{kq})$ is a surjection.
\begin{equation}\label{piSurj}
\begin{tikzcd}[sep=scriptsize]
	{\pi_1(\mathds{1})} && {\pi_1(\MSL)} \\
	&& {\pi_1(\textbf{kq})}
	\arrow[two heads, from=1-1, to=1-3]
	\arrow[two heads, from=1-1, to=2-3]
	\arrow[from=1-3, to=2-3]
\end{tikzcd}
\end{equation}
Let $f_q E$ denote the $q$-th effective cover of a spectrum $E$. R\"ondigs explicitly computed $\pi_1(f_1 \textbf{kq})$, and $\pi_1(f_1 \mathds{1})$.\\[1.000000ex] 
The rows of \cref{obvious1} are exact. Using the fact from \cref{s0MSL} that $s_0(\mathds{1})\cong s_0(\MSL)$, from \cref{obvious1}, we get that the map $\pi_1(f_1\mathds{1}) \to \pi_1(f_1\MSL)$ is surjective.
\begin{equation}\label{obvious1}
\begin{tikzcd}[sep=small]
	{\pi_{2+*,*}s_0\mathds{1}} & {\pi_{1+*,*}f_1\mathds{1}} & {\pi_{1+*,*}\mathds{1}} & {\pi_{1+*,*}s_0\mathds{1}} \\
	\\
	{\pi_{2+*,*}s_0\MSL} & {\pi_{1+*,*}f_1\text{MSL}} & {\pi_{1+*,*}\text{MSL}} & {\pi_{1+*,*}s_0\MSL}
	\arrow["\cong", from=1-1, to=3-1]
	\arrow[two heads, from=1-3, to=3-3]
	\arrow["\cong", from=1-4, to=3-4]
	\arrow[from=1-1, to=1-2]
	\arrow[from=1-2, to=1-3]
	\arrow[from=1-3, to=1-4]
	\arrow[from=3-1, to=3-2]
	\arrow[from=3-2, to=3-3]
	\arrow[from=3-3, to=3-4]
	\arrow[from=1-2, to=3-2]
\end{tikzcd}   
\end{equation}
From the proof of \cite[Theorem 2.5]{roendigs2019remarks}, the map $\pi_1(f_1\mathds{1}) \to \pi_1(f_1\textbf{kq})$ is surjective. From \cref{unitofkqfinal}, we get that $\pi_1(f_1\MSL) \to \pi_1(f_1\textbf{kq})$ is surjective.\\[1.00000ex]
From\cite[Lemma 2.3]{roendigs2019remarks}, the $K^{MW}_*(k)$-module $\pi_1(f_1 \textbf{kq})$ is generated by the image of the first topological Hopf map $\eta^{\textit{top}} \in \pi_{1,0}(\mathds{1})$. This can be given the following presentation
\[K^{MW}_*(k)\{\eta^{\textit{top}}\}/(2\eta^{\textit{top}}, \eta^2\eta^{\textit{top}})\cong \pi_1(f_1 \textbf{kq})\]
From \cite[Theorem 2.5]{roendigs2019remarks}, $\pi_1(f_1 \mathds{1})$ can be given the following presentation.
\[(K^{MW}_*(k) \{\nu\} \oplus K^{MW}_*(k) \{\eta^{\textit{top}}\})/(\eta \nu, 2\eta^{\textit{top}}, \eta^2\eta^{\textit{top}}-12\nu) \cong \pi_1(f_1 \mathds{1})\]
Here $\nu \in \pi_{3,2}(\mathds{1})$ is the second stable Hopf map, as constructed in \cref{newSpvanish}. The element $\nu \in \pi_{3,2}(\mathds{1})$ can be lifted to an element $\nu \in \pi_{3,2}(f_1\mathds{1})$. From \cref{newSpvanish}, the unit map of $\MSL$ takes the element $\nu$ to $0 \in \pi_{3,2}(\MSL)$. Therefore, we get that $\pi_{1}(f_1\MSL)$ can be identified as a quotient of $\pi_{1}(f_1\textbf{kq})$. As, the map $\pi_{1}(f_1\MSL) \to \pi_{1}(f_1\textbf{kq})$ is surjective, we get the second vertical isomorphism in \cref{obvious2}.
\begin{equation}\label{obvious2}
\begin{tikzcd}[sep=small]
	{\pi_{2+*,*}s_0\MSL} & {\pi_{1+*,*}f_1\MSL} & {\pi_{1+*,*}\MSL} & {\pi_{1+*,*}s_0\MSL} \\
	\\
	{\pi_{2+*,*}s_0\textbf{kq}} & {\pi_{1+*,*}f_1\textbf{kq}} & {\pi_{1+*,*}\textbf{kq}} & {\pi_{1+*,*}s_0\text{kq}}
	\arrow["\cong", from=1-1, to=3-1]
	\arrow[two heads, from=1-3, to=3-3]
	\arrow["\cong", from=1-4, to=3-4]
	\arrow[from=1-1, to=1-2]
	\arrow[from=1-2, to=1-3]
	\arrow[from=1-3, to=1-4]
	\arrow[from=3-1, to=3-2]
	\arrow[from=3-2, to=3-3]
	\arrow[from=3-3, to=3-4]
	\arrow["\cong", from=1-2, to=3-2]
\end{tikzcd}  
\end{equation}
The rows in \cref{obvious2} are exact. We use from \cite[Theorem 3.2]{veffhermitian}  the fact that $s_0\mathds{1}\to s_0 \textbf{kq}$ is an isomorphism. The third column in \cref{obvious2} has to be an isomorphism. That finishes our proof.
\end{proof}
\begin{remark}
    Over a field of characteristic $p$, where $p$ is an odd prime, \cref{pi1msloliver} works after inverting the characteristic. In this case we have 
    \[\pi_{1}(\MSL[1/p]) \cong \pi_{1}(\textbf{kq}[1/p])\]
\end{remark}
\subsection{Some slices of $\MSL$}
As the slice filtration is triangulated, the slice functors $s_q$ preserve cofiber sequences. Applying $s_0$ on \cref{skeletal}, we get 
\begin{equation} \label{s0MSLactual}
  \begin{tikzcd}[column sep=tiny, row sep = small]
	{s_0(\MSL) } && {s_0(\MSL \wedge \Sigma^{-2,-1} \Sigma^{\infty} \mathds{P}^2)} && {s_0(\MSL \wedge \Sigma^{-2,-1} \Sigma^{\infty} \mathds{P}^3)} & {...} & {s_0(\MGL)} \\
	&& {s_0(\Sigma^{2,1} \MSL) \cong *} && {s_0(\Sigma^{4,2} \MSL) \cong *}
	\arrow["\cong", hook, from=1-1, to=1-3]
	\arrow["\cong",hook, from=1-3, to=1-5]
	\arrow["\cong",hook, from=1-5, to=1-6]
	\arrow[from=1-6, to=1-7]
	\arrow[from=1-3, to=2-3]
	\arrow[from=1-5, to=2-5]
\end{tikzcd}  
\end{equation}

\begin{corollary}\label{s0MSL}
 As clear from \cref{s0MSLactual}, over any Noetherian, finite dimensional base scheme, \[s_0(\MSL) \cong s_0(\MGL) \cong s_0(\mathds{1}) \] 
 In particular, over any field,
 \[s_0(\MSL) \cong s_0(\mathds{1}) \cong \HZ\]
\end{corollary}

Applying $s_1$ on \cref{skeletal}, we get 
\begin{equation} \label{s1first}\begin{tikzcd}[column sep=tiny,row sep=small]
	{s_1(\MSL) } && {s_1(\MSL \wedge \Sigma^{-2,-1} \Sigma^{\infty} \mathds{P}^2)} && {s_1(\MSL \wedge \Sigma^{-2,-1} \Sigma^{\infty} \mathds{P}^3)} & {...} & {s_1(\MGL)} \\
	&& {s_1(\Sigma^{2,1} \MSL) \cong \Sigma^{2,1} s_0(\MSL)} && {s_1(\Sigma^{4,2} \MSL) \cong *}
	\arrow[hook, from=1-1, to=1-3]
	\arrow["\cong", hook, from=1-3, to=1-5]
	\arrow["\cong", hook, from=1-5, to=1-6]
	\arrow[from=1-6, to=1-7]
	\arrow[from=1-3, to=2-3]
	\arrow[from=1-5, to=2-5]
\end{tikzcd}\end{equation}
From \cref{s1first}, we get that over a field of characteristic $0$, \begin{equation} \label{s1Smashed} s_1(\MSL \wedge \Sigma^{-2,-1} \Sigma^{\infty} \mathds{P}^2) \cong s_1(\MGL) \cong \Sigma^{2,1} \HZ\end{equation}
In the following commutative diagram \cref{s1Actual}
\begin{itemize}
    \item both of the columns are cofiber sequences.
    \item the map $s_1(\eta): s_1(\Sigma^{1,1}\mathds{1}) \cong \Sigma^{1,1} \HZ \to s_1(\mathds{1}) \cong \Sigma^{1,1} \HZ/2$ is the canonical projection map, when we work over characteristic $0$ fields \cite[Lemma 2.30]{annals_stable}.
    \item As mentioned in \cref{SlicesMGL}, $s_1(\textit{cone}(\eta)) \cong s_1(\MGL)$. \\As $\textit{cone}(\eta) \cong \Sigma^{-2,-1} \Sigma^{\infty} \mathds{P}^2$, from \cref{s1Smashed}, $s_1(\textit{cone}(\eta) \wedge \MSL) \cong s_1(\MGL)$.
    \item So, the map in the second row also has to be an isomorphism.
\end{itemize}
\begin{equation} \label{s1Actual}
\begin{tikzcd}[column sep=tiny]
	{\Sigma^{1,1} \HZ\cong \Sigma^{1,1} s_0(\mathds{1}) \cong s_1(\Sigma^{1,1}\mathds{1})\wedge_{s_0(\mathds{1} )}s_0(\MSL )} && {s_1(\Sigma^{1,1}\MSL) \cong \Sigma^{1,1}s_0(\MSL)\cong\Sigma^{1,1}\HZ} \\
	{\Sigma^{1,1}\HZ/2 \cong s_1(\mathds{1})\cong s_1(\mathds{1})\wedge_{s_0(\mathds{1})}s_0(\MSL)} && {s_1(\MSL)} \\
	{s_1(\textit{cone}(\eta))\wedge_{s_0(\mathds{1})}s_0(\MSL)} && {s_1(\textit{cone}(\eta) \wedge \MSL) \cong s_1(\MGL)} \\
	{\Sigma^{2,1} \HZ\cong s_1(\Sigma^{2,1}\mathds{1})\wedge_{s_0(\mathds{1})}s_0(\MSL)} && {s_1(\Sigma^{2,1}\MSL) \cong \Sigma^{2,1}\HZ}
	\arrow["\cong", from=1-1, to=1-3]
	\arrow[from=1-3, to=2-3]
	\arrow["{s_1(\eta)\wedge 1}"', from=1-1, to=2-1]
	\arrow[from=2-1, to=3-1]
	\arrow[from=2-3, to=3-3]
	\arrow[from=2-1, to=2-3]
	\arrow["\cong"', from=3-1, to=3-3]
	\arrow[from=3-1, to=4-1]
	\arrow[from=3-3, to=4-3]
	\arrow["\cong", from=4-1, to=4-3]
\end{tikzcd}\end{equation}
This discussion is summarized as follows.
\begin{corollary}\label{s1MSLfinal}
Over any Noetherian, finite dimensional base scheme \[s_1(\MSL) \cong s_1(\mathds{1}) \]
  In particular, over a field of characteristic $0$,
  \[s_1(\MSL) \cong s_1(\mathds{1}) \cong \Sigma^{1,1}\HZ/2\]
\end{corollary}
\addcontentsline{toc}{section}{Bibliography}

\bibliographystyle{alpha}
\bibliography{references.bib}

\begin{thebibliography}{GRS{\O}12}

\bibitem[AF14]{AsokFaselEta}
Aravind Asok and Jean Fasel.
\newblock Algebraic vector bundles on spheres.
\newblock {\em Journal of Topology}, 7(3):894--926, 2014.

\bibitem[AFW20]{etaRealRealization}
Aravind Asok, Jean Fasel, and Ben Williams.
\newblock Motivic spheres and the image of suslin's hurewicz map.
\newblock {\em Inventiones mathematicae}, 219, 01 2020.

\bibitem[Ana15]{SLprojectiveBundleTheorem}
Alexey Ananyevskiy.
\newblock The special linear version of the projective bundle theorem.
\newblock {\em Compositio Mathematica}, 151(3):461–501, 2015.

\bibitem[Ana20a]{ananyevskiy2019sloriented}
Alexey Ananyevskiy.
\newblock {SL}-oriented cohomology theories.
\newblock In {\em Motivic homotopy theory and refined enumerative geometry. Workshop, Universit\"at Duisburg-Essen, Essen, Germany, May 14--18, 2018}, pages 1--19. Providence, RI: American Mathematical Society (AMS), 2020.

\bibitem[Ana20b]{Ananyevskiy2020ThomII}
Alexey Ananyevskiy.
\newblock Thom isomorphisms in triangulated motivic categories.
\newblock {\em Algebraic \& Geometric Topology}, 2020.

\bibitem[AR{\O}20]{veffhermitian}
Alexey Ananyevskiy, Oliver R{\"o}ndigs, and Paul~Arne {\O}stv{\ae}r.
\newblock On very effective hermitian {$K$}-theory.
\newblock {\em Mathematische Zeitschrift}, 294, 04 2020.

\bibitem[Bac18]{BachmannDUKE}
Tom Bachmann.
\newblock On the conservativity of the functor assigning to a motivic spectrum its motive.
\newblock {\em Duke Mathematical Journal}, 167, 2018.

\bibitem[BH17]{normsinmotivichomotopy}
Tom Bachmann and Marc Hoyois.
\newblock Norms in motivic homotopy theory.
\newblock {\em Ast{\'e}risque}, 2017.

\bibitem[BH21]{bachhopetaperiod}
Tom Bachmann and Michael~J. Hopkins.
\newblock $\eta$-periodic motivic stable homotopy theory over fields, 2021.
\newblock \url{https://arxiv.org/abs/2005.06778}.

\bibitem[B{\O}21]{Bachmann2021TopologicalMF}
Tom Bachmann and Paul~Arne {\O}stv{\ae}r.
\newblock Topological models for stable motivic invariants of regular number rings.
\newblock {\em Forum of Mathematics, Sigma}, 10, 2021.

\bibitem[CF66]{SU_Bordism}
P.~E. Conner and E.E. Floyd.
\newblock {\em Torsion in SU-Bordism}.
\newblock American Mathematical Society, Providence, 1966.

\bibitem[DI05]{cellulardi}
Daniel Dugger and Daniel~C Isaksen.
\newblock Motivic cell structures.
\newblock {\em Algebraic \& Geometric Topology}, 5(2):615--652, jun 2005.

\bibitem[EH16]{3264}
David Eisenbud and Joe Harris.
\newblock {\em 3264 and All That: A Second Course in Algebraic Geometry}.
\newblock Cambridge University Press, 2016.

\bibitem[Ful98]{Fulton}
William Fulton.
\newblock {\em Intersection Theory}.
\newblock Springer New York, NY, 1998.

\bibitem[GRS{\O}12]{coloredOperads}
Javier~J. Guti\'errez, Oliver R\"ondigs, Markus Spitzweck, and Paul~Arne {\O}stv{\ae}r.
\newblock Motivic slices and coloured operads.
\newblock {\em Journal of Topology}, 5(3):727--755, 2012.

\bibitem[Hat02]{Hatcher}
Allen Hatcher.
\newblock {\em Algebraic topology}.
\newblock Cambridge University Press, Cambridge, 2002.
\newblock \url{https://pi.math.cornell.edu/~hatcher/AT/ATpage.html}.

\bibitem[Hau23]{haution_2023}
Olivier Haution.
\newblock Odd rank vector bundles in eta-perodic motivic homotopy theory.
\newblock {\em Journal of the Institute of Mathematics of Jussieu}, page 1–32, 2023.

\bibitem[Hoy15]{Hoyoismalgmot}
Marc Hoyois.
\newblock From algebraic cobordism to motivic cohomology.
\newblock {\em Journal für die reine und angewandte Mathematik (Crelles Journal)}, 2015(702):173--226, 2015.

\bibitem[Hu05]{HU_picard}
Po~Hu.
\newblock On the {P}icard group of the stable {$\A^1$}-homotopy category.
\newblock {\em Topology}, 44(3):609--640, 2005.

\bibitem[Iak19]{MariiaYakersonThesis}
Mariia Iakerson.
\newblock {\em Motivic stable homotopy groups via framed correspondences}.
\newblock PhD thesis, Universit{\"a}t Duisburg-Essen, 2019.
\newblock Available at \url{https://nbn-resolving.org/urn:nbn:de:hbz:464-20190402-142938-8 }.

\bibitem[Jar00]{Jardine2000}
J.F. Jardine.
\newblock Motivic symmetric spectra.
\newblock {\em Documenta Mathematica}, 5:445--553, 2000.

\bibitem[KY14]{splitting_principle}
Satoshi Kondo and Seidai Yasuda.
\newblock The riemann–roch theorem without denominators in motivic homotopy theory.
\newblock {\em Journal of Pure and Applied Algebra}, 218(8):1478--1495, 2014.

\bibitem[Lev08]{coniveau}
Marc Levine.
\newblock The homotopy coniveau tower.
\newblock {\em Journal of Topology}, 1(1):217--267, 2008.

\bibitem[LLP18]{Mirror}
Ivan~Yu Limonchenko, Zhi Lu, and Taras~E. Panov.
\newblock Calabi yau hypersurfaces and su-bordism.
\newblock {\em Proceedings of the Steklov Institute of Mathematics}, 2018.

\bibitem[LPC19]{Limonchenko_2019}
I.~Yu. Limonchenko, T.~E. Panov, and G.~S. Chernykh.
\newblock Su-bordism: structure results and geometric representatives.
\newblock {\em Russian Mathematical Surveys}, 74(3):461, June 2019.

\bibitem[Lur09]{HTT}
Jacob Lurie.
\newblock {\em Higher Topos Theory (AM-170)}.
\newblock Princeton University Press, Princeton, 2009.

\bibitem[Lur12]{HA}
Jacob Lurie.
\newblock Higher algebra, 2012.
\newblock \url{https://people.math.harvard.edu/~lurie/papers/HA.pdf}.

\bibitem[Man18]{mantovani}
Lorenzo Mantovani.
\newblock Localizations and completions in motivic homotopy theory, 2018.
\newblock \url{https://arxiv.org/abs/1810.04134}.

\bibitem[Mau63]{Maunder_1963}
C.~R.~F. Maunder.
\newblock The spectral sequence of an extraordinary cohomology theory.
\newblock {\em Mathematical Proceedings of the Cambridge Philosophical Society}, 59(3):567–574, 1963.

\bibitem[Mil60]{Milnor}
John Milnor.
\newblock On the cobordism ring {$\Omega^*$} and a complex analogue, part {I}.
\newblock {\em American Journal of Mathematics}, 82:505, 1960.

\bibitem[Mor05]{StableAoneConnectivity}
Fabien Morel.
\newblock The stable {$\mathbb{A}^1$} connectivity theorems.
\newblock {\em K-theory}, 35:1--68, 06 2005.

\bibitem[Mor06]{MorelPiZero}
Fabien Morel.
\newblock $\mathds{A}^1$-algebraic topology.
\newblock {\em International Congress of Mathematicians. Vol. II}, pages 1035--1059, 2006.

\bibitem[Mor10]{MorelBook}
Fabien Morel.
\newblock {\em $A^1$-Algebraic Topology over a Field}, volume 2052.
\newblock 11 2010.

\bibitem[MV99]{MV99}
Fabien Morel and Vladimir Voevodsky.
\newblock ${\A}^1$-homotopy theory of schemes.
\newblock {\em Publications Mathématiques de l'IHÉS, Volume 90}, 1999.

\bibitem[MVW06]{motcoh_voevodsky_weibel_mazza}
Carlo Mazza, Vladimir Voevodsky, and Charles Weibel.
\newblock {\em Lecture Notes On Motivic Cohomology}.
\newblock American Mathematical Society, Clay Mathematics Institute, 2006.

\bibitem[Nan24]{MyThesis}
Ahina Nandy.
\newblock {\em An Interpolation between Special and General Linear Algebraic Cobordism}.
\newblock PhD thesis, Universit{\"a}t Osnabr{\"u}ck, 2024.

\bibitem[Nen06]{Gysinnenashev}
Alexander Nenashev.
\newblock Gysin maps in oriented theories.
\newblock {\em Journal of Algebra}, 302(1):200--213, 2006.

\bibitem[N{\O}S09]{motlandweber}
Niko Naumann, Paul {\O}stv{\ae}r, and Markus Spitzweck.
\newblock Motivic landweber exactness.
\newblock {\em Documenta mathematica Journal der Deutschen Mathematiker-Vereinigung}, 14:551--593, 06 2009.

\bibitem[Nov60]{novikovFirst}
S~Novikov.
\newblock Some problems in the topology of manifolds, connected with the theory of thom spaces.
\newblock {\em Doklady Akademii Nauk SSSR}, 1960.

\bibitem[Nov62]{novikov}
S~Novikov.
\newblock Homotopy properties of thom complexes.
\newblock {\em Matematicheskii Sbornik. Novaya Seriya}, 1962.

\bibitem[Pan03]{PaninOriented2}
Ivan Panin.
\newblock Oriented cohomology theories of algebraic varieties.
\newblock {\em K-theory}, 30:265--314, 2003.

\bibitem[PPR08]{PPR}
Ivan Panin, Konstantin Pimenov, and Oliver R{\"o}ndigs.
\newblock {A universality theorem for Voevodsky's algebraic cobordism spectrum}.
\newblock {\em Homology, Homotopy and Applications}, 10(2):211 -- 226, 2008.

\bibitem[PW10]{PW}
Ivan Panin and Charles Walter.
\newblock On the algebraic cobordism spectra {MSL} and {MSp}.
\newblock {\em St. Petersburg Mathematical Journal}, 34, 11 2010.

\bibitem[Qui69]{Quillen}
D~Quillen.
\newblock On the formal group laws of unoriented and complex cobordism theory.
\newblock {\em Bulletin of the American Mathematical Society}, 75(6):1293 -- 1298, 1969.

\bibitem[R{\O}16]{slicesofhermitianktheory}
Oliver R{\"o}ndigs and Paul~Arne {\O}stv{\ae}r.
\newblock {Slices of hermitian $K$–theory and Milnor's conjecture on quadratic forms}.
\newblock {\em Geometry \& Topology}, 20(2):1157 -- 1212, 2016.

\bibitem[Rob14]{robaloThesis}
Marco Robalo.
\newblock {\em Théorie Homotopique Motivique des Espaces non-commutatifs}.
\newblock Phd thesis, Université Montpellier, 2014.

\bibitem[R{\"o}n09]{thetacharacteristic}
Oliver R{\"o}ndigs.
\newblock {T}heta {C}haracteristics and {S}table {H}omotopy {T}ypes of {C}urves.
\newblock {\em The Quarterly Journal of Mathematics}, 61(3):351--362, 02 2009.

\bibitem[R{\"o}n20]{roendigs2019remarks}
Oliver R{\"o}ndigs.
\newblock Remarks on motivic moore spectra.
\newblock In {\em Motivic homotopy theory and refined enumerative geometry. Workshop, Universit\"at Duisburg-Essen, Essen, Germany, May 14--18, 2018}. Providence, RI: American Mathematical Society (AMS), 2020.

\bibitem[R{\"o}n23]{AutomorphismP2}
Oliver R{\"o}ndigs.
\newblock Endomorphisms of the projective plane and the image of the suslin-hurewicz map.
\newblock {\em Inventiones mathematicae}, 232:1161--1194, 2023.

\bibitem[RS{\O}19]{annals_stable}
Oliver R{\"o}ndigs, Markus Spitzweck, and Paul~Arne {\O}stv{\ae}r.
\newblock {The first stable homotopy groups of motivic spheres}.
\newblock {\em Annals of Mathematics}, 189(1):1 -- 74, 2019.

\bibitem[Spi10]{slices_MGL}
Markus Spitzweck.
\newblock {Relations between slices and quotients of the algebraic cobordism spectrum}.
\newblock {\em Homology, Homotopy and Applications}, 12(2):335 -- 351, 2010.

\bibitem[Spi12]{SpitzweckMotCoh}
Markus Spitzweck.
\newblock A commutative {$P^1$}-spectrum representing motivic cohomology over {D}edekind domains {I}.
\newblock {\em M\'emoires de la Soci\'et\'e math\'ematique de France}, 157, 07 2012.

\bibitem[{Sta}24]{stacks-project_Henselization}
The {Stacks project authors}.
\newblock The stacks project.
\newblock \url{https://stacks.math.columbia.edu}, 2024.

\bibitem[Vez00]{vezzosi_BP}
Gabriele Vezzosi.
\newblock Brown-peterson spectra in stable ${\A}^1$-homotopy theory.
\newblock {\em Rendiconti del Seminario Matematico dell 'Universita' di Padova/Mathematical Journal of the University of Padova}, 106, 05 2000.

\bibitem[Voe98]{voevodsky_icm}
Vladimir Voevodsky.
\newblock ${\A}^1$-homotopy theory.
\newblock {\em Documenta Mathematica, Extra Volume ICM I (1998), 579-604}, 1998.

\bibitem[Voe00]{open_problems}
Vladimir Voevodsky.
\newblock Open problems in the motivic stable homotopy theory, {I}.
\newblock {\em Int. Press Lect. Ser.}, 3, 03 2000.

\bibitem[VR{\O}07]{summermotivic}
Vladimir Voevodsky, Oliver R{\"o}ndigs, and Paul~Arne {\O}stv{\ae}r.
\newblock {\em Voevodsky's Nordfjordeid Lectures: Motivic Homotopy Theory}.
\newblock Springer Berlin Heidelberg, Berlin, Heidelberg, 2007.

\bibitem[Wal66]{Wall_1966}
C.~T.~C. Wall.
\newblock Addendum to a paper of conner and floyd.
\newblock {\em Mathematical Proceedings of the Cambridge Philosophical Society}, 62(2):171–175, 1966.

\bibitem[Wen12]{wendt2012examples}
Matthias Wendt.
\newblock More examples of motivic cell structures, 2012.
\newblock \url{https://arxiv.org/abs/1012.0454}.

\bibitem[WW19]{HandbookUnstable}
Kirsten Wickelgren and Ben Williams.
\newblock Unstable motivic homotopy theory.
\newblock {\em Handbook of Homotopy Theory}, 2019.

\bibitem[Yak21]{Yakerson_2021}
Maria Yakerson.
\newblock The unit map of the algebraic special linear cobordism spectrum.
\newblock {\em Journal of the Institute of Mathematics of Jussieu}, 20(6):1905–1930, 2021.

\end{thebibliography}
\end{document}